\documentclass[USenglish, a4paper, oneside, 10pt]{article}
\makeatletter
\let\stdmaketitle\maketitle
\makeatother

\PassOptionsToPackage{nameinlink,capitalize}{cleveref}
\PassOptionsToPackage{noend}{algpseudocode}

\usepackage{mathrsfs}
\usepackage[
arxiv,
alglinenumber,
eqreset=section,
thmreset=section,
]{myPreamble}

\usepackage{amsmath}
\usepackage{amssymb}
\usepackage{amsfonts}
\usepackage{caption}
\usepackage[labelsep=colon]{caption}
\usepackage{stackrel}

\usepackage{commath}
\usepackage{mathtools} 

\usepackage{url}
\usepackage{multirow}
\usepackage{cite}

\usepackage{subcaption}
\usepackage{pgf}
\usepackage{pgfplots}
\usepgfplotslibrary{fillbetween}
\usepackage{tikz}
\usetikzlibrary{arrows,automata}
\usetikzlibrary{intersections}
\usetikzlibrary{calc}
\usetikzlibrary{shapes}
\usetikzlibrary{positioning}
\usetikzlibrary{decorations.text}
\pgfdeclarelayer{background}
\pgfdeclarelayer{foreground}
\pgfsetlayers{background,main,foreground}
\usetikzlibrary{calc,patterns,decorations.pathmorphing,decorations.markings}

\definecolor{mycol}{RGB}{19,48,128}
\usepackage{hyperref} 
\hypersetup{
	hidelinks,
	colorlinks = true,
	citecolor = mycol,
	linkcolor = mycol,
	urlcolor = mycol,
	hypertexnames = true,
	plainpages= false
}

\definecolor{reddishbrown}{HTML}{A52A2A}

\graphicspath{{figures/}}

\newcommand{\inlineitem}[1][]{
	\ifnum\enit@type=\tw@
	{\descriptionlabel{#1}}
	\hspace{\labelsep}
	\else
	\ifnum\enit@type=\z@
	\refstepcounter{\@listctr}\fi
	\quad\@itemlabel\hspace{\labelsep}
	\fi}
\renewcommand{\theenumi}{\arabic{enumi}}

\usepackage{xcolor}
\usepackage{fancyhdr}
\usepackage{graphicx}
\usepackage{amsthm}
\usepackage{commath}
\usepackage{color}
\usepackage{mathtools} 
\usepackage{graphicx}
\usepackage{booktabs}
\usepackage{adjustbox}
\usepackage{tabularx}
\usepackage{float}
\usepackage{svg}
\usepackage{wrapfig}
\usepackage{titlesec}

\setlist{leftmargin=*}
\newcommand{\enumlabel}[1]{\textbf{#1.\arabic*}}
\newcommand{\enumref}[1]{{#1.\arabic*}}

\usepackage{cleveref}
\crefname{section}{\S}{\S\S}
\Crefname{section}{\S}{\S\S}
\newcommand{\Secref}[1]{Section \ref{#1}}

\titlelabel{\thetitle.\quad}

\renewcommand{\figurename}{\textbf{Fig.}}
\renewcommand{\tablename}{\textbf{Table}}

\newcounter{savefootnote}
\newcounter{symfootnote}
\newcommand{\symfootnote}[1]{%
	\setcounter{savefootnote}{\value{footnote}}%
	\setcounter{footnote}{\value{symfootnote}}%
	\ifnum\value{footnote}>8\setcounter{footnote}{0}\fi%
	\let\oldthefootnote=\thefootnote%
	\renewcommand{\thefootnote}{\fnsymbol{footnote}}%
	\footnote{#1}%
	\let\thefootnote=\oldthefootnote%
	\setcounter{symfootnote}{\value{footnote}}%
	\setcounter{footnote}{\value{savefootnote}}%
}

\newcommand{\eps}{\varepsilon}

\newcommand{\calC}{{\mathcal{C}}}
\newcommand{\calE}{{\mathcal{E}}}
\newcommand{\calF}{{\mathcal{F}}}
\newcommand{\calR}{{\mathcal{R}}}

\newcommand{\calQ}{{\mathcal{Q}}}

\newcommand{\calG}{{\mathcal{G}}}
\newcommand{\calH}{{\mathcal{H}}}
\newcommand{\calK}{{\mathcal{K}}}
\newcommand{\calM}{{\mathcal{M}}}
\newcommand{\calN}{{\mathcal{N}}}

\newcommand{\calL}{{\mathcal{L}}}

\allowdisplaybreaks
\newcommand{\rr}{{\mathbb R}}
\newcommand{\rnn}{{\mathbb R_{+}}}
\newcommand{\rp}{{\mathbb R_{++}}}
\newcommand{\sps}[1]{{\mathcal S_{+}^{#1}}}
\newcommand{\spd}[1]{{\mathcal S_{++}^{#1}}}
\newcommand{\cc}[1]{{\mathcal C^{#1}}}
\newcommand{\gsc}[1]{{\calF_{M_{#1},\nu}}}
\newcommand{\scsmooth}[1]{{\mathcal{S}_{M_{#1},\nu}^{\mu}}}
\newcommand{\stt}{\mathop{\, | \,}\nolimits}
\newcommand{\card}{\mathop{\rm card}\nolimits}
\newcommand{\sign}{\mathop{\rm sign}\nolimits}
\newcommand{\diag}{\mathop{\rm diag}\nolimits}
\newcommand{\prox}{\mathop{\rm prox}\nolimits}
\newcommand{\dom}{\mathop{\rm dom}\nolimits}
\newcommand{\epi}{\mathop{\rm epi}\nolimits}
\newcommand{\epis}{\mathop{\rm epi_s}\nolimits}
\newcommand{\epilim}{\mathop{\rm \textrm{e--}\lim}\nolimits}
\newcommand{\epiarrow}{\mathop{\underrightarrow{e}}\nolimits}
\newcommand{\conlim}{\mathop{\rm \textrm{c--}\lim}\nolimits}
\newcommand{\conarrow}{\mathop{\underrightarrow{c}}\nolimits}
\newcommand{\pointlim}{\mathop{\rm \textrm{p--}\lim}\nolimits}
\newcommand{\pointarrow}{\mathop{\underrightarrow{p}}\nolimits}

\newcommand{\hess}[1]{\mathop{#1''}\nolimits}
\newcommand{\third}[1]{\mathop{#1'''}\nolimits}
\newcommand{\gradn}{\mathop{\nabla}\nolimits}
\newcommand{\hessn}{\mathop{\nabla^2}\nolimits}
\newcommand{\thirdn}{\mathop{\nabla^3}\nolimits}
\newcommand{\extended}[1]{\mathop{#1\colon\rr^n\to \rr \cup \{+\infty\}}\nolimits}

\newcommand{\infconv}[2]{\mathop{#1 \square #2}\nolimits}
\newcommand{\infconvdot}[2]{\mathop{#1 \boxdot #2}\nolimits}
\newcommand{\gps}{\mathop{g_s}\nolimits}
\newcommand{\pclsc}[1]{\mathop{\Gamma_0(#1)}\nolimits}
\newcommand{\cf}{{cf.}}
\newcommand{\etc}{{etc.}}
\newcommand{\mapg}[1]{\mathop{G_{\alpha_k g}(#1)}\nolimits}

\usepackage{pgf,tikz}
\usetikzlibrary{arrows}

\makeatletter
\newenvironment{proof*}[1][\proofname]{\par
	\pushQED{\qed}%
	\normalfont \partopsep=\z@skip \topsep=\z@skip
	\trivlist
	\item[\hskip\labelsep
	\itshape
	#1\@addpunct{.}]\ignorespaces
}{%
	\popQED\endtrivlist\@endpefalse
}

\definecolor{mycol}{RGB}{19,48,128}
\usepackage{hyperref} 
\hypersetup{
	hidelinks,
	colorlinks = true,
	citecolor = mycol,
	linkcolor = mycol,
	urlcolor = mycol,
    hypertexnames = true,
    plainpages= false
}

\usepackage{rotating}

\titlelabel{\thetitle.\quad}

\providecommand{\keywords}[1]{\textbf{Keywords:} #1}
\providecommand{\amsclassification}[1]{\textbf{2020 AMS Subject Classification:} #1}

\usepackage{mlmodern}
\makeatletter
\renewcommand{\maketitle}{\bgroup\setlength{\parindent}{0pt}
	\begin{center}
		\textbf{\Large \@title}\vspace*{1.2em}
		
		\textbf{\@author}
	\end{center}
	\egroup
}

\makeatletter
\let\maketitle\stdmaketitle
\makeatother

\title{
	Self-concordant smoothing in proximal quasi-Newton algorithms for large-scale convex composite optimization
	\thanks{
		This work was funded by the European Union (ERC Advanced Research Grant COMPACT, No. 101141351). Views and opinions expressed are however those of the authors only and do not necessarily reflect those of the European Union or the European Research Council. Neither the European Union nor the granting authority can be held responsible for them.
	}
}

\author{
	Adeyemi D.~Adeoye\thanks{
		IMT School for Advanced Studies Lucca, 
		Piazza S.Francesco, 19, 55100 Lucca, Italy.
		{\it E-mails:} \{adeyemi.adeoye,~alberto.bemporad\}@imtlucca.it}
	\and
	Alberto Bemporad\footnotemark[2]
}
\date{}

\begin{document}
\maketitle

\begin{abstract} 
We introduce a notion of \emph{self-concordant smoothing} for minimizing the sum of two convex functions, one of which is smooth and the other nonsmooth. The key highlight is a natural property of the resulting problem's structure that yields a variable metric selection method and a step length rule especially suited to proximal quasi-Newton algorithms. Also, we efficiently handle specific structures promoted by the nonsmooth term, such as $\ell_{1}$-regularization and group lasso penalties. A convergence analysis for the class of proximal quasi-Newton methods covered by our framework is presented. In particular, we obtain guarantees, under standard assumptions, for two algorithms: \texttt{Prox-N-SCORE} (a proximal Newton method) and \texttt{Prox-GGN-SCORE} (a proximal generalized Gauss-Newton method). The latter uses a low-rank approximation of the Hessian inverse, reducing most of the cost of matrix inversion and making it effective for overparameterized machine learning models. Numerical experiments on synthetic and real data demonstrate the efficiency of both algorithms against state-of-the-art approaches. A Julia implementation is publicly available at \url{https://github.com/adeyemiadeoye/SelfConcordantSmoothOptimization.jl}.
\end{abstract}
\\[1ex]
\amsclassification{
65K05; 
90C06; 
49M15 
}

\section{Introduction}
We consider the composite optimization problem
\begin{align}
	\min\limits_{x\in\rr^n} \calL(x) \coloneqq f(x) + g(x), \label{eq:prob}
\end{align}
where $f$ is a smooth, convex loss function and $g$ is a closed, proper, convex (nonsmooth) regularization function. Several optimization problems in engineering, machine learning, and finance can be written in the form \eqref{eq:prob}, including sparse signal recovery, image processing, compressed sensing, and most classification and regression tasks in machine learning. Proximal gradient algorithms are arguably the most widely used methods for such problems (see \cite{combettes2011proximal} and the references therein for a comprehensive treatment). These algorithms handle the nonsmooth term $g$ efficiently by employing its \emph{proximal} operator, which is typically assumed to be computable at low cost.  Among other notable approaches is the \emph{partial smoothing} framework of \cite{beck2012smoothing}, which reassesses early \emph{full} smoothing methods that iteratively replace $g$ with a smooth approximation, \eg, \cite{ben2006smoothing, bertsekas2009nondifferentiable}, and instead smooths only one component of $g$ while leaving the other unchanged. The setting they consider has the form $g(x)=\calR(x)+\Omega(x)$, where $\calR$ is typically a scaled $\ell_1$-norm ($\beta\|x\|_{1}$) that promotes sparsity, and $\Omega$ encodes additional structure (group, fused, \etc). The main motivation remains to use the proximal operator of the (unchanged) nonsmooth component so that the intended structures are preserved. Notably, \emph{fast} proximal gradient schemes \cite{nesterov1983method,nesterov2005smooth,tseng2008accelerated,beck2009fast} have become standard for solving such problems.

Although such first-order schemes outperform subgradient or bundle methods \cite{beck2012smoothing}, they often yield only modest solution accuracy \cite{patrinos2013proximal}. Incorporating second-order information typically improves both convergence speed and solution accuracy. The main drawback is the computational cost associated with full Hessian inverse. Several prior works \cite{becker2012quasi,lee2014proximal,patrinos2014forward,tran2015composite,stella2017forward,themelis2018forward} therefore incorporate approximate second-order information into proximal gradient schemes to mimic the performance of true proximal Newton methods. To ensure global convergence, many of these schemes rely on line search or trust-region procedures, which introduce additional computational cost. Some other works avoid such safeguards by imposing extra structure on the smooth term $f$. For instance, in the convex case \cite{tran2015composite} assumes that $f$ is self-concordant; this assumption yields efficient step length and correction rules but confines the approach to settings where these conditions hold (see \cite{mishchenko2023regularized,rodomanov2021greedy}). In contrast, we propose a step length selection rule specifically for proximal quasi-Newton methods; it is derived from a self-concordance-like structure inherent in our scheme and does not demand that $f$ or $g$ be self-concordant.

In particular, we \emph{regularize}\footnote{In this work, we use ``regularization'' and ``smoothing'' interchangeably but use ``regularization'' to emphasize explicit addition of a smooth function (a smooth approximation of the \emph{nonsmooth part} of the problem) to the \emph{smooth part} of the problem.} problem \eqref{eq:prob} by a second smooth function $g_s$, resulting in the following problem:
\begin{align}
	\min\limits_{x\in\rr^n} \calL_s(x) \coloneqq f(x) + \gps(x;\mu) + g(x), \label{eq:partialsmooth-prob}
\end{align}
where\footnote{We occasionally write $g_s(x)$ instead of $g_s(x;\mu)$ to refer to the same function.} $\gps$ is a self-concordant, epi-smoothing function for $g$ with a positive smoothing parameter $\mu$ (see \defnref{def:smooth-function}). By construction (see \Secref{sec:scsmoothing}), the functions $g$ and $\gps$ do not conflict; therefore, efficient proximal schemes can be used to iteratively solve problem \eqref{eq:partialsmooth-prob} and, for a suitable choice of $\mu$, recover the solution of the original formulation \eqref{eq:prob} (see Sections~\ref{sec:prox-N} and \ref{sec:experiments}). The smooth regularizer $\gps$ serves two main algorithmic purposes in this work. First, it provides an adaptive step length selection method analogous to the Newton-decrement framework but without requiring any self-concordance information about $f$. Second, its Hessian has a simple diagonal structure that can be exploited as a variable metric to scale the proximal operator of $g$ efficiently. As a result, the regularization enhances both the solvability of the smooth part and the handling of the nonsmooth component.

While our development does not rely on any particular structure of $g$, \Secref{sec:structured} shows how known structures can be incorporated, extending the approach to a broader class of structured penalties. For lasso and multi-task regression with structured sparsity, we relate Nesterov's smoothing \cite{nesterov2005smooth} to our framework and combine the ``prox-decomposition'' property of $g$ with the smoothness of $g_s$, thereby enabling straightforward treatment of such structures.

Most notably, three observations are vital to the development of our algorithmic framework:
\begin{enumerate}
	\item For many practical optimization problems, \eg, those that arise in modern machine learning, proximal Newton methods enjoy powerful convergence guarantees but are often computationally prohibitive. This motivates the use of proximal quasi-Newton schemes that use low-rank updates at each iteration (see \Secref{sec:ggnscore}).
	
	\item The infimal-convolution smoothing technique employed to construct $\gps$ uncovers a structure that falls within the self-concordant regularization (SCORE) framework of \cite{adeoye2021sc,adeoye2023score}. Consequently, we can devise an efficient adaptive step length rule for proximal quasi-Newton algorithms without requiring the original problem to be self-concordant. In other words, our development extends SCORE so that it accommodates nonsmooth regularizers while preserving problem-specific structure.
	
	\item The notion of \emph{epi-smoothing functions} introduced in \cite{burke2013epi} permit a principled combination of the smooth regularizer $g_{s}$ with proximal algorithms that handles the nonsmooth term $g$, assuming an efficient method exists for evaluating its proximal operator.  Moreover, the diagonal Hessian of $g_{s}$ serves as a natural variable metric for the resulting scheme, enabling efficient computation of the scaled proximal operator.
\end{enumerate}

Burke and Hoheisel \cite{burke2013epi,burke2017epi} developed the notion of \emph{epi-smoothing} for studying several epigraphical convergence (\emph{epi-convergence}) properties for convex composite functions by combining the infimal convolution smoothing framework due to Beck and Teboulle \cite{beck2012smoothing} with the idea of \emph{gradient consistency} due to Chen \cite{chen2012smoothing}. The key variational analysis tool used throughout their development is the \emph{coercivity} of the class of regularization kernels studied in \cite{beck2012smoothing}. In particular, they establish the close connection between epi-convergence of the regularization functions and supercoercivity of the regularization kernel. Then, based on the above observations, we synthesize this idea with the notion of \emph{self-concordant regularization} \cite{adeoye2021sc,adeoye2023score} to propose two proximal-type algorithms, viz., \texttt{Prox-N-SCORE} (Algorithm \ref{alg:NewtonSCOREProx}) and \texttt{Prox-GGN-SCORE} (Algorithm \ref{alg:GGNSCOREProx}), for convex composite minimization.

\paragraph*{Paper organization.} The rest of this paper is organized as follows: In \Secref{sec:notations}, we present some notations and background on convex analysis. In \Secref{sec:smoothing}, we establish our self-concordant smoothing notion with some properties and results. We describe our proximal quasi-Newton scheme in \Secref{sec:prox-N}, and present the \texttt{Prox-N-SCORE} and \texttt{Prox-GGN-SCORE} algorithms. In \Secref{sec:structured}, we describe an approach for handling specific structures promoted by the nonsmooth function $g$ in problem \eqref{eq:prob}, and propose a practical extension of the so-called \emph{prox-decomposition} property of $g$ for the self-concordant smoothing framework, which has certain in-built smoothness properties. Convergence properties of the \texttt{Prox-N-SCORE} and \texttt{Prox-GGN-SCORE} algorithms are studied in \Secref{sec:convergence}. In \Secref{sec:experiments}, we present some numerical simulation results for our proposed framework with an accompanying Julia package\footnote{\url{https://github.com/adeyemiadeoye/SelfConcordantSmoothOptimization.jl}}, and compare the results with other state-of-the-art approaches. Finally, we give a concluding remark and discuss prospects for future research in \Secref{sec:conclusions}.

\section{Notation and preliminaries}\label{sec:notations}
We denote by $\bar{\rr}\coloneqq \rr \cup \{-\infty, +\infty\}$ the set of extended real numbers. The sets $\rnn \coloneqq [0, +\infty[$ and $\rp \coloneqq \rnn \backslash \{0\}$, respectively, denote the set of nonnegative and positive real numbers. Let $\extended{g}$ be an extended real-valued function. The \emph{(effective) domain} of $g$ is given by $\dom{g} \coloneqq \{x\in \rr^n \stt g(x) < +\infty \}$ and its \emph{epigraph} (resp., \emph{strict epigraph}) is given by $\epi{g} \coloneqq \{(x, \gamma)\in \rr^n \times \rr \stt g(x) \le \gamma \}$ (resp., $\epis{g} \coloneqq \{(x, \gamma)\in \rr^n \times \rr \stt g(x) < \gamma \}$). Given $\gamma \in \rp$, the $\gamma$-sublevel set of $g$ is $\Gamma_\gamma(g) \coloneqq \{x\in \rr^n \colon g(x) \le \gamma\}$. The standard inner product between two vectors $x,y\in\rr^n$ is denoted by $\left<\cdot,\cdot\right>$, that is,  $\left<x,y\right> \coloneqq x^\top y$, where $x^\top$ is the transpose of $x$.

For an $n\times n$ matrix $H$, we write $H\succ 0$ (resp., $H\succeq 0$) to say $H$ is positive definite (resp., positive semidefinite). The sets $\sps{n}$ and $\spd{n}$, respectively, denote the set of $n\times n$ symmetric positive semidefinite and symmetric positive definite matrices. The set $\set{\diag(v)\stt v\in\rr^n}$, where $\diag\colon \rr^n \to \rr^{n\times n}$, defines the set of all diagonal matrices in $\rr^{n\times n}$. Matrix $I_d$ denotes the $d\times d$ identity matrix. We denote by $\card(\calG)$, the cardinality of a set $\calG$. For any two functions $f$ and $g$, we define $(f\circ g)(\cdot) \coloneqq f(g(\cdot))$. We denote by $\calC^k(\rr^n)$, the class of $k$-times continuously-differentiable functions on $\rr^n$, $k\in \rnn$. If the $p$-th derivatives of a function $f\in \cc{k}(\rr^n)$ is $L_f$-Lipschitz continuous on $\rr^n$ with $p\le k$, $L_f\in \rnn$, we write $f\in \calC_{L_f}^{k,p}(\rr^n)$. The notation $\norm{\cdot}$ stands for the standard Euclidean (or $2$-) norm $\norm{\cdot}_2$. We define the weighted norm induced by $H \in \spd{n}$ by $\norm{x}_H \coloneqq \left<Hx, x\right>^{\frac{1}{2}}$, for $x\in \rr^n$. An Euclidean ball of radius $r$ centered at $\bar{x}$ is denoted by $\mathcal{B}_r(\bar{x}) \coloneqq \{x \in \rr^n \stt \norm{x - \bar{x}} \le r \}$. Associated with a given $H \in \spd{n}$, the (Dikin) ellipsoid of radius $r$ centered at $\bar{x}$ is defined by $\calE_r(\bar{x})\coloneqq \{x\in \rr^n \stt \norm{x-\bar{x}}_H \le r\}$. We define the spectral norm $\|A\| \equiv \|A\|_2$ of a matrix $A\in \rr^{m\times n}$ as the square root of the maximum eigenvalue of $A^\top A$, where $A^\top$ is the transpose of $A$.

A convex function $\extended{g}$ is said to be \emph{proper} if $\dom{g} \ne \emptyset$. The function $g$ is said to be lower semicontinuous (lsc) at $y$ if $g(y) \le \liminf\limits_{x\to y} g(x)$; if it is lsc at every $y\in \dom g$, then it is said to be lsc on $\dom g$. We denote by $\pclsc{D}$ the set of proper convex lsc functions from $D \subseteq \rr^n$ to $\rr \cup \{+\infty\}$. Given $g\in \cc{3}(\dom{g})$, we respectively denote by $g'(t)$, $g''(t)$ and $g'''(t)$ the first, second and third derivatives of $g$, at $t\in \rr$, and by $\gradn _x g(x)$, $\hessn _x g(x)$, and $\thirdn _x g(x)$ the gradient, Hessian and third-order derivative tensor of $g$, respectively, at $x\in \rr^n$; if the variables with respect to which the derivatives are taken are clear from context, the subscripts are omitted. If $\hessn g(x) \in \spd{n}$ for a given $x\in \rr^n$, then the \emph{local} norm $\norm{\cdot}_x$ with respect to $g$ at $x$ is defined by $\norm{d}_x\coloneqq \left<\hessn g(x)d,d\right>^{1/2}$, the weighted norm of $d$ induced by $\hessn g(x)$. The associated dual norm is denoted $\norm{v}_x^{\diamond}  \coloneqq \left<\hessn g(x)^{-1}v,v\right>^{1/2}$, for $v\in \rr^n$. The subdifferential $\partial g \colon \rr^n \to 2^{\rr^n}$ of a proper function $\extended{g}$ is defined by $x \mapsto \set{u \in \rr^n \stt (\forall y \in \rr^n) \, \left<y-x,u\right> + g(x) \le g(y)}$, where $2^{\rr^n}$ denotes the set of all subsets of $\rr^n$. The function $g$ is said to be subdifferentiable at $x\in \rr^n$ if $\partial g(x) \ne \emptyset$; the subgradients of $g$ at $x$ are the members of $\partial g(x)$.

We define set convergence in the sense of Painlev\'{e}-Kuratowski. Let $\mathbb{N}$ denote the set of natural numbers. Let $\{C_k\}_{k\in \mathbb{N}}$ be a sequence of subsets of $\rr^n$. The outer limit of $\{C_k\}_{k\in \mathbb{N}}$ is the set
\begin{align*}
	\limsup_{k\to \infty} C_k \coloneqq \set{x \in \rr^n \stt \exists\{k_j\}_{j\in\mathbb{N}}, \exists \{x_j\}_{j\in\mathbb{N}} \, \forall j,x_k \in C_k, \{x_k\} \rightarrow x},
\end{align*}
and its inner limit is
\begin{align*}
	\liminf_{k\to \infty} C_k \coloneqq \set{x \in \rr^n \stt \exists x_k \in C_k \colon \{x_k\} \rightarrow x, \forall k\in \mathbb{N}}.
\end{align*}
The limit $C$ of $\{C_k\}_{k\in \mathbb{N}}$ exists if its outer and inner limits coincide, and we write
\begin{align*}
	C = \lim\limits_{k\to\infty} C_k \coloneqq \limsup_{k\to \infty} C_k = \liminf_{k\to \infty} C_k.
\end{align*}

We say that a function $\extended{g}$ is coercive if $\liminf\limits_{\norm{x} \to \infty}g(x) = +\infty$, and supercoercive if $\liminf\limits_{\norm{x} \to \infty}\frac{g(x)}{\norm{x}} = +\infty$. The sequence $\{g_k\}$ of functions $g_k\colon \rr^n \to \bar{\rr}$ is said to epi-converge to the function $g\colon \rr^n \to \bar{\rr}$ if $\lim\limits_{k\to\infty}\epi{g_k} = \epi{g}$; it is said to continuously converge to  $g$ if for all $x\in \rr^n$ and $\{x_k\}\to x$, we have $\lim\limits_{k\to\infty}g_k(x_k) = g(x)$; and it converges pointwise to $g$ if for all $x\in\rr^n$, $\lim\limits_{k\to\infty}g_k(x) = g(x)$. Epi-convergence, continuous convergence, and pointwise convergence of $\{g_k\}$ to $g$ are respectively denoted by $\epilim g_k=g$ (or $g_k\epiarrow g$), $\conlim g_k=g$ (or $g_k\conarrow g$), and $\pointlim g_k=g$ (or $g_k\pointarrow g$).

The conjugate (or Fenchel conjugate, or Legendre transform, or Legendre-Fenchel transform) $\extended{g^\star }$ of a function $\extended{g}$ is the mapping $y \mapsto \sup\limits_{x\in\rr^n}\left\{\left<x, y\right> - g(x)\right\}$, and its biconjugate is $g^{\star \star} = (g^\star )^\star $.

\section{Self-concordant regularization}\label{sec:smoothing}
This section introduces the concept of self-concordant smoothing, which provides structures that can be exploited in composite optimization problems. We begin by presenting the definition of generalized self-concordant functions, as given in \cite{sun2019generalized}.
\begin{definition}[Generalized self-concordant function on $\rr$]\label{def:gsc-r}
	A univariate convex function $g\in \cc{3}(\dom g)$, with $\dom g$ open, is said to be $(M_g,\nu)$-generalized self-concordant, with $M_g\in\rnn$ and $\nu\in \rp$, if
	\begin{align*}
		\abs{\third{g}(t)} \le M_g \hess{g}(t)^\frac{\nu}{2}, \qquad \forall t \in \dom g.
	\end{align*}
\end{definition}
\begin{definition}[Generalized self-concordant function on $\rr^n$ of order $\nu$]\label{def:gsc-rn}
	A convex function $g\in \cc{3}(\dom g)$, with $\dom g$ open, is said to be $(M_g,\nu)$-generalized self-concordant of order $\nu\in \rp$, with $M_g\in\rnn$, if $\forall x \in \dom{g}$
	\begin{align*}
		\abs{\left<\thirdn{g}(x)[v]u,u\right>} \le M_g\norm{u}_x^2\norm{v}_x^{\nu-2}\norm{v}^{3-\nu}, \qquad \forall u,v \in \rr^n,
	\end{align*}
	where $\nabla^3{g}(x)[v] \coloneqq \lim\limits_{t\to 0} \left\{\left(\nabla^2{g}(x + tv)-\nabla^2{g}(x)\right)/{t}\right\}$ is the third directional derivative of $g$.
\end{definition}
Note that for an $(M_g,\nu)$-generalized self-concordant function $g$ defined on $\rr^n$, the univariate function $\varphi\colon\rr\to\rr\cup \{+\infty\}$ defined by $\varphi(t)\coloneqq g(x + tv)$ is $(M_g,\nu)$-generalized self-concordant for every $x,v\in \dom g$ and $x + tv \in \dom g$. This provides an alternative definition for the generalized self-concordant function on $\rr^n$.

A key observation from Definitions \ref{def:gsc-r} and \ref{def:gsc-rn} is the possibility to extend the theory beyond the case $\nu = 3$ and $u=v$ originally presented in \cite{nesterov1994interior}. This observation, for instance, allowed the authors in \cite{bach2010self} to introduce a \emph{pseudo} self-concordant framework, in which $\nu = 2$, for the analysis of logistic regression. In a recent development, the authors in \cite{ostrovskii2021finite} identified a new class of pseudo self-concordant functions and showed how these functions may be slightly modified to make them \emph{standard} self-concordant (i.e., where $M_g=2, \nu=3, u=v$), while preserving desirable structures. With such generalizations, and stemming from the idea of \emph{Newton decrement} \cite{nesterov1994interior}, we propose new step length selection techniques for proximal quasi-Newton methods from the self-concordant regularization framework of this section. We denote by $\gsc{g}$ the class of $(M_g,\nu)$-generalized self-concordant functions, with generalized self-concordant parameters $M_g\in\rnn$ and $\nu\in \rp$.
\begin{definition}[Self-concordant smoothing function]\label{def:smooth-function}
	We say that the parameterized function $g_s\colon \rr^n \times \rp \to \rr$ is a self-concordant smoothing function for $g\in\pclsc{\rr^n}$ if the following two conditions are satisfied:
	\begin{enumerate}[label=\enumlabel{SC}, ref=\enumref{SC}]
		\item $\epilim\limits_{\mu\downarrow 0} \gps(x; \mu) = g(x)$.\label{ass:sc1}
		\item $\gps(x; \mu) \in \gsc{g}$. \label{ass:sc2}
	\end{enumerate}
\end{definition}
We denote by $\scsmooth{g}$ the set of self-concordant smoothing functions for a function $g\in\pclsc{\rr^n}$, that is, $\scsmooth{g}\coloneqq\set{g_s\colon \rr^n \times \rp \to \rr \stt g_s \epiarrow g,~ g_s \in \gsc{g}}$.

\subsection{Self-concordant regularization via infimal convolution}\label{sec:scsmoothing}
Next, we present key elements of smoothing through infimal convolution, which includes the Moreau-Yosida regularization process as a special case in defining the (scaled) proximal operator.
\begin{definition}[Infimal convolution]
	Let $g$ and $h$ be two functions from $\rr^n$ to $\rr \cup \{+\infty\}$. The infimal convolution (or ``inf-convolution'' or ``inf-conv'')\footnote{Also sometimes called ``epigraphic sum'' or ``epi-sum'', as its operation yields the (strict) \emph{epigraphic sum} $\epi g + \epi h$ \cite[p. 93]{hiriart2004fundamentals}.} of $g$ and $h$ is the function $\infconv{g}{h}\colon \rr^n \to \bar{\rr}$ defined by
	\begin{align}
		(\infconv{g}{h})(x) = \inf\limits_{w\in\rr^n}\left\{g(w)+h(x-w)\right\}. \label{eq:infconv}
	\end{align}
\end{definition}
The infimal convolution of $g$ with $h$ is said to be \emph{exact at $x \in\dom g$} if the infimum \eqref{eq:infconv} is attained. It is \emph{exact} if it is exact at each $x \in\dom g$, in which case we write $\infconvdot{g}{h}$. Of utmost importance about the inf-conv operation in this paper is its use in the approximation of a function $g\in\pclsc{\rr^n}$; that is, the approximation of $g$ by its infimal convolution with a member $h_\mu(\cdot)$ of a parameterized family $\calH \coloneqq \{h_\mu \, \stt \, \mu \in \rp\}$ of (regularization) kernels. In more formal terms, we recall the notion of inf-conv regularization in \defnref{def:infconvreg}. For $h\in \pclsc{\rr^n}$ and $\mu \in \rp$, we define the function $\extended{h_\mu}$ by the \emph{epi-multiplication} operation\footnote{It is easy to show that $h_\mu^\star  = \mu h^\star $.}
\begin{align}
	h_\mu(\cdot) \coloneqq \mu h\left(\frac{\cdot}{\mu}\right), \quad \mu \in \rp. \label{eq:hmu}
\end{align}
\begin{definition}[Inf-conv regularization]\label{def:infconvreg}
	Let $g$ be a function in $\pclsc{\rr^n}$. Define
	\begin{align*}
		\calH \coloneqq \left\{(x, w) \mapsto h_\mu(x-w) \stt \mu \in \rp \right\}
	\end{align*}
	a parameterized family of regularization kernels. The inf-conv regularization process of $g$ with $h_\mu \in \calH$  is given by $(\infconv{g}{h}_\mu)(x)$, for any $x\in \rr^n$.
\end{definition}
The operation of the inf-conv regularization generalizes the Moreau-Yosida regularization process in which case, $h_\mu(\cdot)=\norm{\cdot}^2/(2\mu)$ or, with a scaled norm, $h_\mu(\cdot)=\norm{\cdot}_Q^2/(2\mu)$ for some $Q\in\spd{n}$. The Moreau-Yosida regularization process provides the value function of the proximal operator associated with a function $g\in \pclsc{\rr^n}$. This leads us to the definition of the scaled proximal operator.
\begin{definition}[Scaled proximal operator]
	The scaled proximal operator of a function $g\in \pclsc{\rr^n}$, written $\prox_{\alpha g}^{Q}(\cdot)$, for $\alpha \in \rp$ and $Q\in \spd{n}$, is defined as the unique point in $\dom g$ that satisfies
	\begin{align*}
		(\infconv{g}{\psi_\alpha})(x) = g(\prox_{\alpha g}^{Q}(x))+\psi_\alpha(x - \prox_{\alpha g}^{Q}(x)),
	\end{align*}
	where $\psi_\alpha(\cdot)\coloneqq\norm{\cdot}_Q^2/(2\alpha)$. That is, $\prox_{\alpha g}^{Q}(x) \coloneqq \argmin_{w\in \rr^n} \{g(w) + \psi_\alpha(x-w)\}$.
\end{definition}
A key property of the scaled proximal operator is its \emph{nonexpansiveness}; that is, the property that (see, \eg, \cite{rockafellar1976monotone} and \cite[Lemma 2]{tran2015composite})
\begin{align}\label{eq:nonexpansive}
	\norm{\prox_{\alpha g}^Q(x) - \prox_{\alpha g}^Q(y)}_Q \le \norm{x-y}_{Q^{-1}},
\end{align}
for all $x,y\in\rr^n$.

In the sequel, we assume that the regularization kernel function $h$ is of the form
\begin{align}
	h(x) = \sum_{i=1}^{n} \phi(x^{(i)}), \label{eq:hseparable}
\end{align}
where $\phi$ is a univariate \emph{potential function}. We are now left with the question of what properties we need to hold for $\phi$ such that $\infconv{g}{h_\mu}$ produces $g_s$ satisfying the self-concordant smoothing conditions \ref{ass:sc1} -- \ref{ass:sc2}. To this end, we assume that $\phi$ satisfies the following:
\begin{enumerate}[label=\enumlabel{K}, ref=\enumref{K}]
	\item $\phi$ is supercoercive. \label{ass:k1}
	\item $\phi \in \gsc{\phi}$. \label{ass:k2}
\end{enumerate}
Many functions that appear in different settings naturally exhibit the structures in conditions \ref{ass:k1} -- \ref{ass:k2}. For example, the ones belonging to the class of \emph{Bregman/Legendre functions} introduced by Bauschke and Borwein \cite{bauschke1997legendre} (see also \cite{de1986relaxed} for a related characterization of the class of \emph{Bregman functions}). In the context of proximal gradient algorithms for solving \eqref{eq:prob}, the recent paper \cite{bauschke2017descent} enlists these functions as satisfying the new descent lemma (a.k.a \emph{descent lemma without Lipschitz gradient continuity}) which the paper introduced. We summarize examples of these regularization kernel functions on different domains in \tablename~\ref{tab:kernel-functions}. We extract practical examples on $\rr$ for the smoothing of the $1$-norm and the indicator functions.
\begin{remark}\label{rem:bounded-below}
	Suppose that $\dom h$ is a nonempty bounded subset of $\rr^n$, for example, if $\phi \in \pclsc{\dom \phi}$, then since we have that $g\in\pclsc{\dom g}$ is bounded below as it possesses a continuous affine minorant (in view of \cite[Theorem 9.20]{bauschke2011convex}), the less restrictive condition that $\phi$ is coercive sufficiently replaces the condition \ref{ass:k1}. In other words, the key convergence notion presented later holds similarly for the resulting function $\infconv{g}{h_\mu}$ in this case. Particularly, we get that $\infconv{g}{h_\mu}$ in this case is exact, finite-valued and locally Lipschitz continuous (see, \eg, \cite[Proposition 3.6]{burke2017epi}) making it fit into our algorithmic framework.
\end{remark}
\begin{remark}\label{rem:supercoercive}
	Whenever the supercoercivity condition is difficult to check (and the condition in \remref{rem:bounded-below} does not hold), two possibilities exist according to \cite[Proposition 3.9]{burke2017epi}: (1) If $h\in\pclsc{\rr^n}$ is such that $\infconv{g}{h_\mu} \epiarrow g$, and $g\in \pclsc{\rr^n}$ is supercoercive, then $h$ is necessarily supercoercive; (2) If, however, $g\in \pclsc{\rr^n}$ is not supercoercive, then we can find some $h\in \pclsc{\rr^n}$ that is not supercoercive but for which $\infconv{g}{h_\mu} \epiarrow g$.
\end{remark}
In light of \remref{rem:bounded-below} and \remref{rem:supercoercive}, our examples in \tablename~\ref{tab:kernel-functions} include both coercive and supercoercive functions. In either case, we have $\phi \in \gsc{\phi}$. We keep the supercoercivity condition to emphasize other realizable properties of $\infconv{g}{h_\mu}$.
\begin{table*}[t!]
	\centering
	\small
	\caption{Examples of regularization kernel functions for self-concordant smoothing, and their generalized self-concordant parameters $M_\phi$ and $\nu$ (see \defnref{def:gsc-r}).}
	\begingroup
	\begin{tabular*}{\textwidth}{@{\extracolsep\fill}ccccc}
		\toprule
		$\phi(t)$ & $\dom{\phi}$ & $M_{\phi}$ & $\nu$ & Remark \\
		\midrule
		$\frac{1}{p}\sqrt{1+p^2\abs{t}^2}-1$, $p\in\rp$ & $\rr$ & $2$ & $2.6$ & $p=1$ \\
		$\frac{1}{2}\left[\sqrt{1+4t^2}-1+\log\left(\frac{\sqrt{1+4t^2}-1}{2t^2}\right)\right]$ & $\rr$ & $2\sqrt{2}$ & $3$ & Ostrovskii \& Bach \cite{ostrovskii2021finite} \\
		$\frac{1}{2}t^2$ & $\rr$ & $0$ & $3$ & ``Energy'' \\
		$\frac{1}{p}\abs{t}^p$, $p\in (1,2)$ & $\rnn$ & $4$ & $6$ & $p=1.5$ \\
		$\log(1 + \exp(t))$ & $\rr$ & $1$ & $2$ & ``Logistic'' \\
		$t\log t - t$ & $[0,+\infty]$ & $1$ & $4$ & ``Boltzmann-Shannon'' \\
		$\begin{cases}
			\frac{1}{2}(t^2 - 4t + 3), \quad \text{if} \,\,\, t\le 1\\
			-\log t, \quad \text{otherwise}
		\end{cases}$ & $\rr$ & $4$ & $3$ & De Pierro \& Iusem \cite{de1986relaxed} \\
		\bottomrule
	\end{tabular*}
	\label{tab:kernel-functions}
	\endgroup
\end{table*}
\paragraph*{Examples.}
For some functions $g$ and $h_\mu$, there exists a closed form solution to $\infconv{g}{h_\mu}
$. On the other hand, if one gets that $\infconv{g}{h_\mu}=\infconvdot{g}{h_\mu}\in \pclsc{\rr^n}$, \eg, as a result of \propref{thm:exact}, then knowing in this case that
\begin{align}
	\infconv{g}{h_\mu} = (g^\star  + h_\mu^\star )^\star ,\label{eq:inf-conv-dual}
\end{align}
we can efficiently estimate $\infconv{g}{h_\mu}$ using \emph{fast} numerical schemes (see, \eg, \cite{lucet1997faster}). The structure of $h$ implies $\gps$ can be expressed in terms of a corresponding univariate function $\varphi\colon\rr\to\rr\cup \{+\infty\}$ by defining $\varphi_s(t;\mu) \coloneqq (\infconv{\varphi}{h_\mu})(t)$, and then
\begin{align*}
	\gps(x;\mu) = \sum_{i=1}^{n} \varphi_s(x^{(i)};\mu).
\end{align*}
In the following, we provide examples of such $\varphi_s$ for some $\phi \in \gsc{\phi}$.
\paragraph*{Infimal convolution of $\|\cdot\|_1$ with $h_\mu$.} In the first two examples, we consider $g(x)= \|x\|_1$.
\begin{example}\label{ex:infconv1}
	Let $p=1$ in $\phi(t) = \frac{1}{p}\sqrt{1+p^2\abs{t}^2}-1$, with $\dom \phi = \rr$. Then,
	\begin{align*}
		\varphi_s(t;\mu) = \frac{\mu^{2}-\mu  \sqrt{\mu^{2}+t^{2}}+t^{2}}{\sqrt{\mu^{2}+t^{2}}}.
	\end{align*}
\end{example}
\begin{figure*}[t!]
	\centering
	\subfloat{%
		\resizebox*{7cm}{!}{\includegraphics{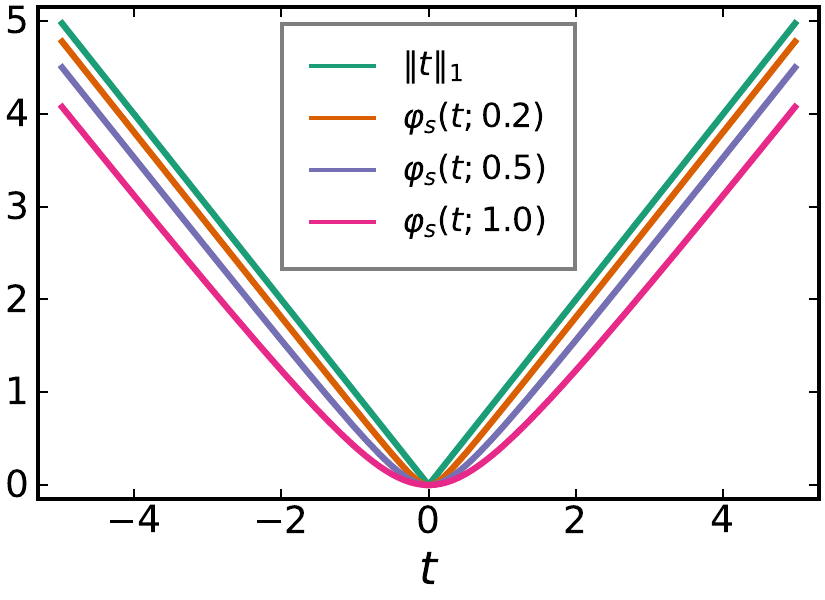}}}\hspace{5pt}
	\subfloat{%
		\resizebox*{7cm}{!}{\includegraphics{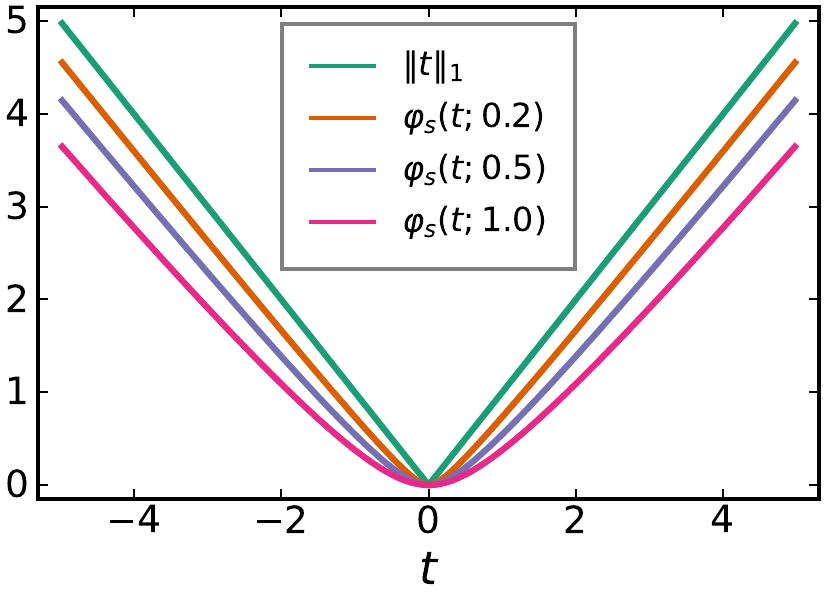}}}
	\caption{Generalized self-concordant smoothing of $\|\cdot\|_1$ with $\phi(t) = \sqrt{1+\abs{t}^2}-1$ (left) and $\phi(t) = \frac{1}{2}\left[\sqrt{1+4t^2}-1+\log\left(\frac{\sqrt{1+4t^2}-1}{2t^2}\right)\right]$ (right). The smooth approximation is shown for $\mu=0.2,0.5,1.0$.} \label{fig:g-smooth}
\end{figure*}
\begin{example}\label{ex:infconv2}
	$\phi(t) = \frac{1}{2}\left[\sqrt{1+4t^2}-1+\log\left(\frac{\sqrt{1+4t^2}-1}{2t^2}\right)\right]$, with $\dom \phi = \rr$:
	\begin{align*}
		\varphi_s(t;\mu) =& \frac{\sqrt{\mu^{2}+4 t^{2}}}{2} -\frac{\mu}{2}\left[1+\log (2) - \log \left(\frac{2t -\sqrt{\mu^{2}+4 t^{2}}+\mu}{t}\right) -\log \left(\frac{2t + \sqrt{\mu^{2}+4 t^{2}}-\mu}{t}\right)\right].
	\end{align*}
\end{example}
\paragraph*{Infimal convolution of $\delta_C(x)$ with $h_\mu$.} In the next example, we consider $g(x)= \delta_C(x)$, where $C \coloneqq \{x\in\rr^n \stt l\le x \le u\}$ and
\begin{align*}
	\delta_C(x) \coloneqq \begin{cases}
		0, \quad &\text{if } x\in C,\\
		+\infty, \quad &\text{otherwise}.
	\end{cases}
\end{align*}

\begin{example}\label{ex:infconv3}
	Let $g(x) = \delta_C(x)$, and consider
	\begin{align*}
		\phi(t) = \begin{cases}
			\frac{1}{2}(t^2 - 4t + 3), \quad &\text{if} \,\,\, t\le 1\\
			-\log t, \quad &\text{otherwise},
		\end{cases}
	\end{align*}
	with $\dom \phi = \rr$. We have
	\begin{align*}
		\varphi_s(t;\mu) = \begin{cases}
			\frac{1}{2\mu}\left(l - t +3\mu \right)\left(l - t + \mu \right), \quad &\text{if}\,\,\, l \ge t-\mu\\
			\mu\log(\mu)-\mu\log(t - l), \quad &\text{otherwise}.
		\end{cases}
	\end{align*}
\end{example}

The next two results characterize the functions $h$ and $h_\mu$ defined by supercoercive and generalized self-concordant kernel functions.
\begin{lemma}\label{thm:hsc}
	Let $\phi \in \pclsc{\rr}$ be a function from $\rr$ to $\rr \cup \{+\infty\}$, and let the function $\extended{h}$ be defined by $h(x) \coloneqq \sum_{i=1}^{n} \lambda_i\phi(x^{(i)})$ with $x^{(i)}\in \dom{\phi}$, $\lambda_i \in \rp$, $i=1,2,\ldots,n$. Then the following properties hold:
	\begin{enumerate}
		\item $h \in \pclsc{\rr^n}$.
		\item $h$ is supercoercive if and only if $\phi$ is supercoercive on its domain.
		\item If $\phi \in \gsc{\phi}$, where $M_{\phi} \in \rnn$ and $\nu \ge 2$, then $h(x)$ is well-defined on $\dom{h} = \{\dom{\phi}\}^n$, and $h(x) \in \gsc{h}$, with $M_h \coloneqq \max\{\lambda_i^{1- \frac{\nu}{2}}M_{\phi} \stt 1\le i \le n\} \in \rnn$.\label{item:hsciii}
	\end{enumerate}
\end{lemma}
\begin{proof}
	\begin{enumerate}
		\item This statement is a direct consequence of \cite[Corollary 9.4, Lemma 1.27 and Proposition 8.17]{bauschke2011convex}.
		\item Follows directly from the definition of supercoercivity.
		\item $h(\cdot) \in \gsc{h}$ with $M_h \coloneqq \max\{\lambda_i^{1- \frac{\nu}{2}}M_{\phi} \stt 1\le i \le n\} \in \rnn$ follows from \cite[Proposition 1]{sun2019generalized}.
	\end{enumerate}
\end{proof}

\begin{proposition}[Self-concordance of $h_\mu$]\label{thm:hmugsc}
	Suppose the conditions of \lemref{thm:hsc} hold such that the function $\extended{h}$ defined by \eqref{eq:hseparable} is $(M_h,\nu)$-generalized self-concordant. Let $A\in \rr^{n\times n}$ be a diagonal matrix defined by $A\coloneqq \diag(\frac{1}{\mu})$ such that $h(\frac{x}{\mu})\equiv h(Ax)$ is an affine transformation of $h(x)$. Then the following properties hold:
	\begin{enumerate}
		\item If $\nu \in (0,3]$, then $h_\mu\in\gsc{}$ with $M=n^{\frac{3-\nu}{2}}\mu^{\frac{\nu}{2}-2}M_h$.
		\item If $\nu >3$, then $h_\mu\in\gsc{}$ with $M=\mu^{4-\frac{3\nu}{2}}M_h$.
	\end{enumerate}
\end{proposition}
\begin{proof}
	\begin{enumerate}
		\item We have $\norm{A} = \frac{\sqrt{n}}{\mu}$. By \cite[Proposition 2(a)]{sun2019generalized}, $h(\frac{x}{\mu}) \in \gsc{}$ with $M= \norm{A}^{3-\nu}M_h$. In view of \lemref{item:hsciii}, the scaling $h(\frac{\cdot}{\mu})\mapsto \mu h(\frac{\cdot}{\mu})$ gives $M\mapsto \mu^{1-\frac{\nu}{2}}M$. The result follows.
		\item The value $\mu^2\in \rp$ corresponds to the unique eigenvalue of $A^\top A$. By \cite[Proposition 2(b)]{sun2019generalized}, $h(\frac{x}{\mu}) \in \gsc{}$ with $M= \mu^{3-\nu}M_h$. The result follows as in Item (i) above.
	\end{enumerate}
\end{proof}

The next result concerns the epi-convergence of smoothing via infimal convolution under the condition of supercoercive regularization kernels in $\pclsc{\rr^n}$.
\begin{lemma}{\cite[Theorem 3.8]{burke2017epi}}\label{thm:epiconvergence}
	Let $g, h \in \pclsc{\rr^n}$ with $h$ supercoercive and $0\in\dom{h}$. Let $h_\mu$ be defined as in \eqref{eq:hmu}. Then the following hold:
	\begin{enumerate}
		\item $\epilim\limits_{\mu\downarrow 0}\{g^\star  + \mu h^\star \} = g^\star $.
		\item $\epilim\limits_{\mu\downarrow 0}\{\infconv{g}{h_\mu}\} = g$. \label{thm:lemepi}
		\item  If $h(0)\le 0$, we have $\pointlim\limits_{\mu\downarrow 0}\{\infconv{g}{h_\mu}\} = g$.
	\end{enumerate}
\end{lemma}
The main argument for the notion of epi-convergence in optimization problems is that when working with functions that may take infinite values, it is necessary to extend traditional convergence notions by applying the theory of \emph{set convergence} to epigraphs in order to adequately capture local properties of the function (through a resulting calculus of smoothing functions), which on the other hand may be challenging due to the \emph{curse of differentiation} associated with nonsmoothness. We refer the interested reader to \cite[Chapter 7]{rockafellar2009variational} for further details on the notion of epi-convergence, and to \cite{stromberg1994study,burke2013epi,burke2017epi} for extended results on epi-convergent smoothing via infimal convolution.

In the following, we highlight key properties of the infimal convolution of $g\in \pclsc{\rr^n}$ with $h_\mu$ satisfying $h\in\gsc{h}$.
\begin{proposition}\label{thm:g-properties}
	Let $g,h \in \pclsc{\rr^n}$. Suppose further that $h$ is $(M_h,\nu)$-generalized self-concordant and supercoercive, and define $g_s \coloneqq \infconv{g}{h_\mu}$ for all $\mu \in \rp$. Then the following hold:
	\begin{enumerate}
		\item $\infconv{g}{h_\mu}=\infconvdot{g}{h_\mu}\in \pclsc{\rr^n}$. \label{thm:exact}
		\item\label{thm:gs-gsc} $g_s \in \scsmooth{g}$ with
		\begin{align*}
			M_g =
			\begin{cases}
				n^{\frac{3-\nu}{2}}\mu^{\frac{\nu}{2}-2}M_h, \qquad &\text{if } \nu \in (0,3],\\
				\mu^{4-\frac{3\nu}{2}}M_h, \qquad &\text{if } \nu >3.
			\end{cases}
		\end{align*}
		\item $g_s$ is locally Lipschitz continuous.\label{thm:g-lip}
	\end{enumerate}
\end{proposition}
\begin{proof}
	First, as an immediate consequence of \cite[Lemma 1.28, Lemma 1.27 and Proposition 8.17]{bauschke2011convex}, we have $h_\mu\in\pclsc{\rr^n}$.
	\begin{enumerate}
		\item Follows immediately from \cite[Proposition 12.14]{bauschke2011convex}.
		\item By Item \ref{thm:exact}, $g_s=\infconvdot{g}{h_\mu}\in \pclsc{\rr^n}$. As a consequence of \cite[Proposition 12.14]{bauschke2011convex}, we have
		\begin{align*}
			g_s(x,\mu) = \min\limits_{w\in\rr^n}\left\{g(w)+h_\mu(x-w)\right\},
		\end{align*}
		and $g_s \epiarrow g$ (by \cite[Theorem 11.34]{rockafellar2009variational}).
		In view of \cite[Proposition 7.2]{rockafellar2009variational}, for $x\in \dom g$ and
		\begin{align*}
			w_\mu(x) \in \argmin_{w\in \rr^n}\set{g(w) + h_\mu(x-w)} \ne \emptyset,
		\end{align*}
		$g_s \epiarrow g$ implies that $g_s(x,\mu) \to g(x)$ for at least one sequence $w_\mu(x) \to x$. Hence, we have
		\begin{align*}
			(\infconv{g}{h_\mu})(x) = g(w_\mu(x)) + h_\mu(x - w_\mu(x)).
		\end{align*}
		And, given $h\in \gsc{h}$, we have by \propref{thm:hmugsc} that $h_\mu$ is $(M_g,\nu)$-generalized self-concordant, where $M_g$ is given by
		\begin{align*}
			M_g =
			\begin{cases}
				n^{\frac{3-\nu}{2}}\mu^{\frac{\nu}{2}-2}M_h, \qquad &\text{if } \nu \in (0,3],\\
				\mu^{4-\frac{3\nu}{2}}M_h, \qquad &\text{if } \nu >3.
			\end{cases}
		\end{align*}
		Hence, $h_\mu\in\cc{3}(\dom g)$, and by \cite[Proposition 18.7/Corollary 18.8]{bauschke2011convex}, noting that higher-order derivatives are defined inductively in this sense \cite[Deﬁnition 2.54, Remark 2.55]{bauschke2011convex}, we deduce
		\begin{align*}
			\abs{\left<\thirdn (\infconv{g}{h_\mu})(x)[v]u,u\right>} &= \abs{\left<\thirdn h_\mu(x - w_\mu(x))[v]u,u\right>}, \quad \forall u,v \in \dom g,
		\end{align*}
		and similarly for the second-order derivatives. By definition, the univariate function
		\begin{align}
			\varphi(t)\coloneqq h_\mu(u_1 + tv_1), \label{eq:univariatesc}
		\end{align}
		is $(M_g,\nu)$-generalized self-concordant, for every $u_1,v_1 \in \dom g$. That is, $\forall t\in\rr$,
		\begin{align*}
			\abs{\third{\varphi}(t)} \le M_g\hess{\varphi}(t)^{\frac{\nu}{2}},
		\end{align*}
		which concludes the proof with $u_1\equiv x$, $v_1\equiv w(\frac{x}{\mu})$ and $t \equiv -\mu$ in \eqref{eq:univariatesc}.
		\item Following the arguments in Items (i) and (ii) above, $w_\mu$ (and hence $g_s$) is finite-valued (see also \cite[Lemma 4.2]{burke2013epi}). Then the Lipschitz continuity of $g_s$ near some $\bar{x}\in\dom g$ follows from the convexity of $g_s$ (see \cite[Example 9.14]{rockafellar2009variational}; see also \cite[Proposition 3.6]{burke2017epi}).
	\end{enumerate}
\end{proof}

\section{A proximal quasi-Newton scheme}\label{sec:prox-N}
Our notion of self-concordant smoothing developed in the previous section is motivated by algorithmic purposes. Notably, we have established the epi-convergence of $\gps\in \gsc{g}$ to $g\in \pclsc{\rr^n}$ under suitable conditions, which plays a critical role in the optimization problem \eqref{eq:partialsmooth-prob} in a global sense. We next characterize the optimal solution set of \eqref{eq:partialsmooth-prob} using the notion of $\varepsilon$-optimality with respect to \eqref{eq:prob}. We define $\varepsilon$-$\argmin g \coloneqq \{x \stt g(x) \le \inf g + \varepsilon\}$ to be the set of points that minimize the function $g$ up to a tolerance $\varepsilon \in \rnn$. For our approach, it suffices to state the following about the set of minimizers of $\gps$.
\begin{proposition}\label{thm:epi-minimizer}
	Fix any $\mu \in \rp$. Suppose $g \in \pclsc{\rr^n}$ and $\gps \in \scsmooth{g}$. Then a minimizer of $\gps$ is $\varepsilon^\mu$-optimal for $g$ with $\varepsilon^\mu \in \rnn$.
\end{proposition}
\begin{proof}
	From \propref{thm:g-lip}, we have that, for any $\bar{x}\in \dom g$, $g_s \epiarrow g$ implies there is at least one sequence $w_\mu(\bar{x}) \to \bar{x}$. By the (super)coercivity of $\gps$, the level set $\{x\in \rr^n \stt \gps(x;\mu) \le \hat{\alpha}\}$ at $\hat{\alpha} \in \rr$ is bounded and contained in a compact set $C$ such that $w_\mu(\bar{x}) \in C$. Let $w_\mu(\bar{x}) \in \varepsilon^\mu$-$\argmin \gps \subseteq C$ (with $\mu \in \rp$ fixed). Then, since $g_s \epiarrow g$, we get from \cite[Theorem 7.31(b)]{rockafellar2009variational} that $g(\bar{x}) \le \inf g + \varepsilon^\mu$. Hence, $\bar{x} \in \varepsilon^\mu$-$\argmin g$. Finally, $w_\mu(\bar{x}) \in \varepsilon^\mu$-$\argmin g$ necessarily follows from \cite[Theorem 7.33]{rockafellar2009variational}.
\end{proof}

\propref{thm:epi-minimizer}, along with the observation in \cite[Theorem 7.37]{rockafellar2009variational}, suggests that a proximal algorithm can provide a solution to \eqref{eq:partialsmooth-prob}, which also solves \eqref{eq:prob} with a high accuracy. Hence, the proximal method effectively handles the nonsmooth part of the problem, while our regularization approach enhances both the solvability of the smooth part of the original problem and improves the handling of the nonsmooth part through the choice of the variable metric. For the optimization problem \eqref{eq:partialsmooth-prob}, we assume the following:
\begin{enumerate}[label=\enumlabel{P}, ref=\enumref{P}]
	\item $f$ is convex and $f\in\calC_{L_f}^{2,2}(\rr^n)$.\label{ass:p1}
	\item $\rho_1 I_n \le \hessn f(x^\star ) \le L_1I_n$, $\rho_2 I_n \le \hessn g_s(x^\star ) \le L_2I_n$ at a locally optimal solution $x^\star $ of \eqref{eq:partialsmooth-prob} with $L_1 \ge \rho_1 \in \rp$ and $L_2\ge \rho_2 \in \rp$.\label{ass:p2}
	\item $g\in \pclsc{\rr^n}$.\label{ass:p3}
	\item $\gps\in \scsmooth{g}$.\label{ass:p4}
\end{enumerate}
In particular, we consider $\gps(x;\mu)\coloneqq \infconv{g}{h_\mu}$, where $h$ is a suitable regularization kernel for self-concordant smoothing of $g$ in the sense of \Secref{sec:smoothing}.

Proximal quasi-Newton algorithms for solving \eqref{eq:partialsmooth-prob} consist in minimizing a sequence of \emph{upper approximation} of $\calL_s$ obtained by summing the nonsmooth part $g(x_k)$ and a local quadratic model of the smooth part $q(x_k)\coloneqq f(x_k) + \gps(x_k)$ near $x_k$. That is, for $x\in\dom{\calL}\equiv \dom f \cap \dom g$, we iteratively define
\begin{subequations}
	\begin{align}
		\hat{q}_k(x) &\coloneqq q(x_k) + \left<\gradn q(x_k),x-x_k\right> + \frac{1}{2}\norm{x-x_k}_{Q}^2,\label{eq:proxnewtonstep0}\\
		\hat{m}_k(x) &\coloneqq \hat{q}_k(x) + g(x),
	\end{align}\label{eq:proxnewtonmodel}
\end{subequations}
where $Q\in \spd{n}$, and then solve the subproblem
\begin{align}
	\delta_k \in \argmin_{d\in\rr^n} \hat{m}_k(x_k+d), \label{eq:subproblem}
\end{align}
for a proximal quasi-Newton search direction $\delta_k$. Our characterization of the optimality conditions for \eqref{eq:partialsmooth-prob} in this section, particularly the flexibility in the choice of the variable metric $Q$, is well-motivated by the class of \emph{cost approximation (CA) methods} \cite{patriksson1998cost}. This leads to a novel approach for selecting $\{x_k\}$ from the sequence of iterates $\{\delta_k\}$. The necessary optimality conditions for \eqref{eq:partialsmooth-prob} are defined by
\begin{align}
	0\in \gradn q(x^\star ) + \partial g(x^\star ), \label{eq:optimality}
\end{align}
for $x^\star \in \dom \calL$. To find points $x^\star $ satisfying \eqref{eq:optimality}, CA methods, as the name implies, iteratively approximate $\gradn q(x_k)$ by a \emph{cost approximating mapping} $\Phi\colon \rr^n \to \rr^n$, taking into account the fixed approximation error term $\Phi(x_k) - \gradn q (x_k)$. That is, a point $d$ is sought satisfying
\begin{align}
	0\in \Phi(d) + \partial g(d) + \gradn q(x_k) - \Phi(x_k). \label{eq:ca-optimality}
\end{align}
Let $\Phi$ be the gradient mapping of a continuously differentiable convex function $\psi\colon \rr^n\to \rr$. A CA method iteratively solves the subproblem
\begin{align}
	\min\limits_{d\in\rr^n} \set{\psi(d) + q(x_k) + g(d) - \psi(x_k) + \left<\gradn q(x_k) - \gradn \psi(x_k), d-x_k\right>}.\label{eq:ca-prob}
\end{align}
A step is then taken in the direction $\delta_k - x_k$, namely
\begin{align}
	x_{k+1} = x_k + \alpha_k(\delta_k - x_k),\label{ca-update}
\end{align}
where $\delta_k$ solves \eqref{eq:ca-prob} and $\alpha_k\in \rp$ is a step length typically computed via a line search such that an appropriately selected \emph{merit function} is sufficiently decreased along the direction $\delta_k - x_k$.
\begin{remark}\label{rem:steplength}
	Evaluating the merit function too many times can be impractical. One way to mitigate this issue for large-scale problems is to incorporate ``predetermined step lengths'' into the solution scheme of \eqref{eq:ca-prob}. This allows us to update $x_k$ as $x_{k+1}\equiv\delta_k$. However, methods that use this approach do not generally yield a monotonically decreasing sequence of objective values. Instead, convergence is characterized by a metric that measures the distance from iteration points to the set of optimal solutions \cite{patriksson1993unified}.
\end{remark}
We discuss next a new proximal quasi-Newton scheme that compromises between minimizing the objective values and decreasing the distance from iteration points to the set of optimal solutions as specified by a curvature-exploiting variable metric.
\subsection{Variable metric and adaptive step length selection}
A very nice feature of the CA framework is that it can help, for instance, through the specific choice of $\Phi$, to efficiently utilize the original problem's structure---a practice which is particularly useful when solving medium- to large-scale problems. This feature fits directly into our self-concordant smoothing framework. We notice that \eqref{eq:ca-prob} gives \eqref{eq:subproblem} with the following choice of $\psi$:
\begin{align}
	\psi(\cdot) = \frac{1}{2}\norm{\cdot}_Q^2, \qquad Q\in\spd{n}.\label{eq:cost-approx-term}
\end{align}
In this case, the optimality conditions and our assumptions give
\begin{align}
	(Q-\gradn q)(x_k) \in (Q + \partial g)(d),
\end{align}
which leads to
\begin{align}
	\delta_k = \prox_g^{Q}(x_k - Q^{-1}\gradn q(x_k)).
\end{align}
In the proximal quasi-Newton scheme, $Q$ may be the Hessian of $q(x_k)$ or its (low-rank) approximation. Although a diagonal structure of $Q$ is often desired in this case due to its ease of implementation, we most likely discard relevant curvature information, especially when $q$ is not assumed to be separable. Our consideration in this work entails the following characterization of the optimality conditions:
\begin{align}
	(H_k-\gradn q)(x_k) \in (Q_k + \partial g)(d),\label{eq:optimality-conditions}
\end{align}
where $H_k$ may be the Hessian, $\hessn q(x_k) \equiv H_k^f + H_k^g$, of $q$ or its approximation, where $H_k^f\equiv \hessn f(x_k)$, $H_k^g\equiv \hessn g_s(x_k;\mu)$, and $Q_k \in \spd{n}$. Specifically, we set $Q_k = H_k^g$ in \eqref{eq:optimality-conditions} and propose the following step update formula:
\begin{align}
	x_{k+1} = \prox_{\alpha_k g}^{H_k^g}(x_k - \bar{\alpha}_k H_k^{-1}\gradn q(x_k)),\label{eq:proxnewton-general}
\end{align}
where $\bar{\alpha}_k\in\rp$ results from \emph{damping} the quasi-Newton steps.
\begin{algorithm}
	\caption{\texttt{Prox-N-SCORE} (A proximal Newton algorithm)}
	\label{alg:NewtonSCOREProx}
	\begin{algorithmic}[1]
		\Require $x_0\in\rr^n$, problem functions $f$, $g$, self-concordant smoothing function $g_s\in \scsmooth{g}$, $\alpha\in(0,1]$
		\For{$k=0,\ldots$}
		\State $\mathrm{grad}_k\gets \gradn f(x_k) + \gradn g_s(x_k)$
		\State $H_k^g \gets \hessn g_s(x_k)$; $\eta_k \gets \norm{\gradn g_s(x_k)}_{H_k^{g^{-1}}}$ \Comment{Note: $H_k^g$ is diagonal}
		\State $\bar{\alpha}_k = \frac{\alpha}{1 + M_g\eta_k}$
		\State $H_k \gets \hessn f(x_k)+H_k^g$; Solve for $\Delta_k$: $H_k \Delta_k = \mathrm{grad}_k$
		\State $x_{k+1} \gets \prox_{\alpha g}^{H_k^g}(x_k - \bar{\alpha}_k\Delta_k)$
		\EndFor
	\end{algorithmic}
\end{algorithm}

The validity of this procedure in the present scheme may be seen in the interpretation of the proximal operator $\prox_{g}\left(x^+\right)$ for some $x^+\in\dom g$ as compromising between minimizing the function $g$ and staying close to $x^+$ (see \cite[Chapter 1]{parikh2014proximal}). When scaled by, say, $H_k^g$, ``closeness'' is quantified in terms of the metric induced by $H_k^g$, and we want the proximal steps to stay close (as much as possible) to the Newton iterates relative to, say, $\norm{\cdot}_{H_k^g}$. To see this, we note that in view of the fixed-point characterization \eqref{eq:ca-prob} via CA methods, we may interpret proximal quasi-Newton algorithms as a fixation of the error term $\gradn \psi - \gradn q$ at some point in $\dom q \cap \dom g$. Let us fix some $\bar{x} \in \dom q \cap \dom g$ and introduce the operator $E_{\bar{x}}$ defined by
\begin{align}
	E_{\bar{x}}(z) \coloneqq \hessn q(\bar{x})z - \bar{\alpha} \gradn q(z),\label{eq:error-term1}
\end{align}
where $0<\bar{\alpha}\le\alpha\le1$. Set $Q = Q_k\in\spd{n}$ arbitrary in \eqref{eq:cost-approx-term}. We aim to exploit the structure in $g_s$ (and $\hessn g_s$), so we define an operator $\xi_{\bar{x}}(Q_k,\cdot)$ to quantify the error between $\hessn g_s$ and $Q_k$ as follows:
\begin{align}
	\xi_{\bar{x}}(Q_k,z) \coloneqq (\hessn g_s(\bar{x})-Q_k)(z - x_k).\label{eq:error-term2}
\end{align}
We provide a local characterization of the optimality conditions for \eqref{eq:ca-prob} in terms of $E_{\bar{x}}$ and $\xi_{\bar{x}}$ in the next result.
\begin{proposition}\label{thm:optimality}
	Let the operators $E_{\bar{x}}$ and $\xi_{\bar{x}}(Q_k,\cdot)$ be defined by \eqref{eq:error-term1} and \eqref{eq:error-term2}, respectively. Then the optimality conditions for \eqref{eq:ca-prob} with $\psi(\cdot) = \frac{1}{2}\norm{\cdot}_{Q_k}^2$ are locally characterized in terms of $E_{\bar{x}}$ and $\xi_{\bar{x}}(Q_k,\cdot)$ by
	\begin{align}
		E_{\bar{x}}(x_k) + \xi_{\bar{x}}(Q_k, d) \in \hessn g_s(\bar{x})d + \alpha\partial g(d). \label{eq:new-optimality}
	\end{align}
	More precisely, \eqref{eq:optimality-conditions} holds with $Q_k=\hessn g_s(\bar{x})$ whenever $\bar{x}$ is the unique optimizer satisfying \eqref{eq:new-optimality} at a local solution $d$ of \eqref{eq:ca-prob}.
\end{proposition}
\begin{proof}
	As $\gps$ satisfies the property in \ref{ass:sc1}, it holds that \cite[Lemma 3.4]{burke2013epi}
	\begin{align}
		\limsup_{\substack{x\to\bar{x}\\\mu\downarrow 0}} \gradn g_s(x;\mu) = \partial g(\bar{x}).\label{eq:grad-consistency}
	\end{align}	
	Hence, by \lemref{thm:epiconvergence} and \cite[Theorem 13.2]{rockafellar2009variational}, there exists $v_g \in \rr^n$, in the \emph{extended sense} of differentiability (see \cite[Definition 13.1]{rockafellar2009variational}), such that
	\begin{subequations}
		\begin{align}
			\limsup_{\substack{x\to\bar{x}\\\mu\downarrow 0}} \gradn g_s(x) = \partial g(\bar{x}) = \{v_g\},\\
			\emptyset \ne \partial g(d) \subset v_g + \hessn g_s(\bar{x})(d-\bar{x}) + o(\|d-\bar{x}\|)\calE_r(\bar{x}). \label{eq:extended-differentiability}
		\end{align}
	\end{subequations}
	Let $x_k$ be in some neighbourhood of $\bar{x}$ and let $\{x_k\} \to \bar{x}$ be generated by an iterative process. By assumption, the differentiable terms in \eqref{eq:extended-differentiability} are convex and the differential operators are monotone. It then holds that
	\begin{align}
		\partial g(d) \subset v_g + \hessn g_s(\bar{x})(d-x_k) + o(\|d-\bar{x}\|)\calE_r(\bar{x}), \label{eq:optimal-differentiability}
	\end{align}
	for all $x_k$ in the neighbourhood of $\bar{x}$. Since differentiability in the extended sense is necessary and sufficient for differentiability in the \emph{classical sense} (see \cite[Definition 13.1 and Theorem 13.2]{rockafellar2009variational}), it holds for some $\mu\in \rp$ that $v_g \equiv \gradn g_s(\bar{x})$ which is defined through:
	\begin{align}
		\gradn g_s(d) = \gradn g_s(\bar{x}) + \hessn g_s(\bar{x})(d-\bar{x}) + o(\|d-\bar{x}\|). \label{eq:classical-differentiability}
	\end{align}
	Consequently, using \eqref{eq:ca-optimality} (with $\Phi = \gradn\psi$), and defining the Dikin ellipsoid $\calE_r(\bar{x})$ in terms of $g_s$ for $r$ small enough, we deduce from \eqref{eq:optimal-differentiability}, \eqref{eq:classical-differentiability} that $Q_k(x_k-d) + \hessn g_s(\bar{x})(x_k-d) - \bar{\alpha}\gradn q(x_k) \in \bar{\alpha}\gradn g_s(\bar{x})$ for $0<\bar{\alpha}\le1$. We assert $\hessn f(\bar{x})(d-\bar{x}) \in \calE_r(\bar{x})$ at a local solution $d$ of \eqref{eq:ca-prob}, and then deduce again from \eqref{eq:optimal-differentiability}, \eqref{eq:classical-differentiability} that $\bar{\alpha}\gradn g_s(\bar{x}) + \hessn g_s(\bar{x})(d-x_k) + \hessn f(\bar{x})x_k \in \alpha \partial g(d)$ holds for $0<\bar{\alpha}\le\alpha\le1$ near $\bar{x}$, whenever $\bar{x}$ is the unique solution $x^\star $ of \eqref{eq:partialsmooth-prob}. As a result, using $q \coloneqq f + g_s$, we get
	\begin{align}
		(\hessn q(\bar{x}) - \bar{\alpha}\gradn q)x_k - \hessn g_s(\bar{x})x_k \in Q_k(d-x_k) + \alpha \partial g(d). \label{eq:ca-optimality-expanded}
	\end{align}
	In terms of $E_{\bar{x}}$ and $\xi_{\bar{x}}(Q_k,\cdot)$, \eqref{eq:ca-optimality-expanded} may be written as \eqref{eq:new-optimality}, which exactly gives \eqref{eq:optimality-conditions} with the choice $Q_k = \hessn g_s(\bar{x})$.
\end{proof}

We consider \emph{damping} the quasi-Newton steps such that
\begin{align}
	\bar{\alpha}_k = \frac{\alpha_k}{1 + M_g\eta_k}, \label{eq:steplength}
\end{align}
where $M_g$ is given by \ref{ass:p4} and $\eta_k \coloneqq \norm{\gradn g_s(x_k)}_{x_k}^{\diamond}$ is the dual norm of $\gradn g_s(x_k)$ with respect to $g_s(x_k)$. Note that the above choice for $\bar{\alpha}_k$, in the context of minimizing generalized self-concordant functions, assumes $\nu\ge2$ (see, \eg, \cite[Equation 12]{sun2019generalized}). Suppose for example $\alpha_k=1$ is fixed and $\nu=3$, then \eqref{eq:ca-optimality-expanded} leads to the standard damped-step proximal quasi-Newton method in the framework of Newton decrement (\cf~\cite{tran2015composite,sun2019generalized}).

By \eqref{eq:hseparable}, $H_k^g$ has a desirable diagonal structure and hence can be cheaply updated from iteration to iteration. This structure provides an efficient way to compute the scaled proximal operator $\prox_{g}^{H_k^g}$, \eg, via the proximal calculus presented in \cite{becker2019quasi} (see \Secref{sec:experiments} for two practical examples). Overall, by exploiting the structure of the problem, precisely
\begin{enumerate}
	\item taking adaptive steps that properly capture the curvature of the objective functions, and
	\item scaling the proximal operator of $g$ by a variable metric $H_k^g$ which has a simple, diagonal structure,
\end{enumerate}
we can adapt to an affine-invariant structure due to the quasi-Newton steps and ensure we remain close to {them} towards convergence.

If we choose $H_k\equiv \hessn q(x_k)$ in \eqref{eq:proxnewton-general}, we obtain a proximal Newton step (see Algorithm \ref{alg:NewtonSCOREProx}):
\begin{align}
	x_{k+1} = \prox_{\alpha_k g}^{H_k^g}(x_k - \bar{\alpha}_k\hessn q(x_k)^{-1}\gradn q(x_k)). \label{eq:proxnewton}
\end{align}
However, $H_k$ may be any approximation of the Hessian of $q$ at $x_k$. In view of \eqref{eq:new-optimality}, this corresponds to replacing the Hessian term $\hessn q(\bar{x})$ in \eqref{eq:error-term1} by the approximating matrix evaluated at $\bar{x}$.
\begin{algorithm}
	\caption{\texttt{Prox-GGN-SCORE} (A proximal generalized Gauss-Newton algorithm)}
	\label{alg:GGNSCOREProx}
	\begin{algorithmic}[1]
		\Require $x_0\in\rr^n$, problem functions $f$, $g$, self-concordant smoothing function $g_s\in \scsmooth{g}$, model $\calM$, input-output pairs $\{u^{(i)},y^{(i)}\}_{i=1}^m$ with $y^{(i)} \in \rr^{n_y}$, $\alpha\in (0,1]$
		\For{$k=0,\ldots$}
		\State $H_k^g \gets \hessn g_s(x_k)$; $\eta_k \gets \norm{\gradn g_s(x_k)}_{H_k^{g^{-1}}}$ \Comment{Note: $H_k^g$ is diagonal}
		\State $\bar{\alpha}_k \gets \frac{\alpha}{1 + M_g\eta_k}$
		\If{$m+n_y \le n$}
		\State Compute $\delta_k^{\mathrm{ggn}}$ via \eqref{eq:ggn-step-approx}
		\Else
		\State Compute $\delta_k^{\mathrm{ggn}}$ via \eqref{eq:ggn-step}
		\EndIf
		\State $x_{k+1} \gets \prox_{\alpha g}^{H_k^g}(x_k + \bar{\alpha}_k\delta_k^{\mathrm{ggn}})$
		\EndFor
	\end{algorithmic}
\end{algorithm}
\subsection{A proximal generalized Gauss-Newton algorithm}\label{sec:ggnscore}
In describing the proximal GGN algorithm, consider first the simple case $g\equiv0$. Then \eqref{eq:proxnewton-general} with $\bar{\alpha}_k=1$ gives exactly the pure Newton direction
\begin{align}
	\delta_k^{\mathrm{ggn}} = -H_k^{-1}\gradn q(x_k).\label{eq:newtonstep}
\end{align}
Now suppose that the function $f$ quantifies a data-misfit or loss between the outputs\footnote{Note that for the sake of simplicity, we assume here $y^{(i)}\in\rr$, but it is straightforward to extend the approach that follows to cases where $y^{(i)}\in\rr^{n_y}$, $n_y>1$.} $\hat{y}^{(i)}$ of a model $\calM(\cdot;x)$ and the expected outputs $y^{(i)}$, for $i=1,2,\ldots,m$, as in a typical machine learning problem, and that $g\ne 0$. Precisely, let $\hat{y}^{(i)} \coloneqq \calM(u^{(i)};x)$, and suppose that $f$ can be written as
\begin{align}
	f(x) = \sum_{i=1}^{m}\ell(y^{(i)},\hat{y}^{(i)}),\label{eq:f-ggn}
\end{align}
where $\ell\colon\rr\times\rr\to\rr$ is a loss function. Define an ``augmented'' Jacobian matrix $J_k\in\rr^{(m+1)\times n}$ by \cite{adeoye2021sc,adeoye2023score}
\begin{align}
	J_k^\top  &\coloneqq \begin{bmatrix}
		\gradn_{x_k}{\hat{y}^{(1)}} & \gradn_{x_k}{\hat{y}^{(2)}} & \cdots & \gradn_{x_k}{\hat{y}^{(m)}} & \gradn{\gps(x_k)}
	\end{bmatrix}. \label{eq:Jaug}
\end{align}
Then GGN approximation of the Newton direction \eqref{eq:newtonstep} gives
\begin{align}
	\delta_k^{\mathrm{ggn}} = -(H_k^f + H_k^g )^{-1}\gradn q \approx -(J_k^\top V_k J_k + H_k^g )^{-1}J_k^\top e_k,
	\label{eq:ggn-step}
\end{align}
where the vector $e_k\coloneqq[l'_{\hat{y}^{(1)}}(y^{(1)},\hat{y}^{(1)}),\ldots,l'_{\hat{y}^{(m)}}(y^{(m)},\hat{y}^{(m)}),1]^\top\in\rr^{m+1}$ defines an augmented ``residual'' term, and $V_k\coloneqq\diag(v_k)$, with $v_k\coloneqq[l''_{\hat{y}^{(1)}}(y^{(1)},\hat{y}^{(1)}),\ldots,l''_{\hat{y}^{(m)}}(y^{(m)},\hat{y}^{(m)}),0]^\top\in\rr^{(m+1)}$. If $m+1<n$ (possibly $m\ll n$), that is, when the model is \emph{overparameterized}, the following equivalent formulation of \eqref{eq:ggn-step} provides a computationally efficient way to compute the GGN search direction \cite{adeoye2021sc,adeoye2023score}:
\begin{align}
	\delta_k^{\mathrm{ggn}} = -H_k^{g^{-1}}J_k^\top(I_m+V_k J_k H_k^{g^{-1}}J_k^\top)^{-1}e_k.\label{eq:ggn-step-approx}
\end{align}
Note that in case the function $g$ (and hence $g_s$) is scaled by some (nonnegative) constant, only the identity matrix $I_m$ may be scaled accordingly. Now using $H_k\equiv J_k^\top V_k J_k + H_k^g$ in \eqref{eq:proxnewton-general} gives the proximal GGN update (see Algorithm \ref{alg:GGNSCOREProx}):
\begin{align}
	x_{k+1} = \prox_{\alpha_k g}^{H_k^g}(x_k + \bar{\alpha}_k\delta_k^{\mathrm{ggn}}), \label{eq:proxggn}
\end{align}
where $\bar{\alpha}_k$ is as defined in \eqref{eq:steplength}.
\section{Structured penalties}\label{sec:structured}
As we have noted, more general nonsmooth problems impose certain structures on the variables that must be handled explicitly by the algorithm. Such situations arise in some Lasso and multi-task regression problems, where problem \eqref{eq:prob} takes the form
\begin{align}
	\min\limits_{x\in\rr^n} f(x) + \underbrace{\calR(x) + \Omega(Cx)}_{g(x)},\label{eq:structured-prob}
\end{align}
where, in addition to $\calR(x)$, the function (\cf~\cite{chen2010graph,chen2012smoothingprox})
\begin{align}
	\Omega(Cx) \coloneqq \max_{u\in \calQ} \langle u, Cx \rangle,
\end{align}
enforces a desired structure of the solution estimates. Here $C:\mathbb{R}^{n}\to\mathbb{V}$ is a linear map into a finite-dimensional vector space $\mathbb{V}$, and $\mathcal{Q}\subseteq\mathbb{V}^{\star}$ is a closed, convex subset of the dual space $\mathbb{V}^{\star}$.

For example, in the sparse-group lasso problem \cite{friedman2010note,simon2013sparse}, $\Omega(Cx) = \gamma \sum_{j\in \calG}\omega_j\|x^{(j)}\|$ induces group level sparsity on the solution estimates and $\calR(x) = \beta\norm{x}_1$ promotes the overall sparsity of the solution, so that the optimization problem is written as
\begin{align}
	\min\limits_{x\in\rr^n} f(x) + \beta\norm{x}_1 + \beta_\calG \sum_{j\in \calG}\omega_j\|x^{(j)}\|,\label{eq:glasso}
\end{align}
where $\beta\in\rp$, $\beta_\calG \in \rp$, $\calG = \{j_k, \ldots, j_{n_g}\}$ is the set of variables groups with $n_g=\card(\calG)$, $x^{(j)} \in \rr^{n_j}$ is the subvector of $x$ corresponding to variables in group $j$ and $\omega_j \in \rp$ is the group penalty parameter. Another example is the graph-guided fused lasso for multi-task regression problems \cite{kim2009multivariate}, where the function $\Omega(Cx)=\beta_\calG\sum_{e=(r,s)\in E,r<s}\tau(\omega_{rs})\abs{x^{(r)} - \sign(\omega_{rs})x^{(s)}}$ encourages a fusion effect over variables $x^{(r)}$ and $x^{(s)}$ shared across tasks through a graph $G\equiv (V,E)$ of relatedness, where $V=\{1,\ldots,n\}$ denotes the set of nodes and $E$ the edges; $\beta_\calG\in \rp$, $\tau(\omega_{rs})$ is a fusion penalty function, and $\omega_{rs}\in \rr$ is the weight of the edge $e=(r,s)\in E$. Here, with $\calR(x) = \beta\norm{x}_1$, $\beta\in\rp$, the optimization problem is written as
\begin{align}
	\min\limits_{x\in\rr^n} f(x) + \beta\norm{x}_1 + \beta_\calG\sum_{e=(r,s)\in E,r<s}\tau(\omega_{rs})\abs{x^{(r)} - \sign(\omega_{rs})x^{(s)}}.\label{eq:gflasso}
\end{align}
In both examples, $C$ is defined so as to encode these additional structures. See \Secref{ex:glasso} for an illustration involving the sparse-group lasso.
\subsection{Structure reformulation for self-concordant smoothing}
The key observation in problems of the form \eqref{eq:structured-prob} is that the function $\Omega(Cx)$ belongs to the class of nonsmooth convex functions that is well-structured for Nesterov's smoothing \cite{nesterov2005smooth} in which a smooth approximation $\Omega_s$ of $\Omega$ has the form\footnote{The reader should not confuse the barrier smoothing technique of, say, \cite{nesterov2011barrier,tran2014barrier}, with the self-concordant smoothing framework of this paper. The self-concordant barrier smoothing techniques, just like Nesterov's smoothing, realize first-order and subgradient algorithms that solve problems of this exact form.}
\begin{align}
	\Omega_s(Cx;\mu) = \max_{u\in \calQ} \left\{\langle u, Cx \rangle - \mu d(u)\right\}, \quad \mu \in \rp,\label{eq:nesterovsmooth}
\end{align}
where $d$ is a \emph{prox-function}\footnote{A function $d_1$ is called a \emph{prox-function} of a closed and convex set $\calQ_1$ if $\calQ_1 \subseteq \dom d_1$, and $d_1$ is continuous and strongly convex on $\calQ_1$ with convexity parameter $\rho_1 \in \rp$ \cite{nesterov2005smooth}.} of the set $\calQ$. Note that Nesterov's smoothing approach assumes the knowledge of the exact structure of $C$. In the sequel, we shall write $\Omega^C(x)\equiv \Omega(Cx)$ or $\Omega_s^C(x;\mu)\equiv \Omega_s(Cx)$, with the superscript ``$C$'' to indicate the function is \emph{structure-aware} via $C$.
\begin{proposition}\label{thm:structured-gsc}
	Let $C\colon \rr^n \to \rr^n$ be a linear map and let $\omega$ be a continuous convex function defined on a closed and convex set $Q\subseteq \dom \omega\subseteq\rr^n$. Further, define
	\begin{align*}
		\tilde{\Omega}(x) \coloneqq \max_{u\in \calQ} \left\{\langle u, Cx \rangle - \omega(u)\right\},
	\end{align*}
	and let $d\coloneqq h^\star $, where $h\colon \rr^n \to \rr$ satisfies $\hessn h \in \spd{n}$ and is of the form \eqref{eq:hseparable} with $\phi$ satisfying \ref{ass:k1} -- \ref{ass:k2} so that $h\in\gsc{h}$ with $\nu \in [3,6)$ if $n>1$ and with $\nu \in (0,6)$ if $n=1$. Then the function
	\begin{align}
		\Omega_s(x;\mu) = \max_{u\in \calQ} \left\{\langle u, Cx \rangle - \omega(u) - \mu d(u)\right\}, \quad \mu \in \rp,
	\end{align}
	is a self-concordant smoothing function for $\tilde{\Omega}(x)$.
\end{proposition}
\begin{proof}
	We follow the approach in \cite[Section 4]{beck2012smoothing}. First note that we can write $\tilde{\Omega}(x)=\Omega(Cx)$, where
	\begin{align*}
		\Omega \coloneqq (\omega + \delta_Q)^\star .
	\end{align*}
	Now, let $\tilde{d}\coloneqq d + \delta_Q$. In view of \cite[Proposition 6]{sun2019generalized}, we have $d, \tilde{d} \in \calF_{M_d,\nu_d}$ where $M_d = M_h$ and $\nu_d = 6-\nu$. Next, define $\tilde{h} \coloneqq (\tilde{d})^\star $. We have
	\begin{align*}
		(\Omega^\star  + \tilde{h}_\mu^\star )^\star (x) &= (\omega + \delta_Q + \mu\tilde{d})^\star (x)\\
		&= \max_{u\in \calQ}\left\{\langle u, x\rangle - \omega(u) -\mu d(u)\right\},
	\end{align*}
	which is precisely $(\infconv{\tilde{\Omega}}{h_\mu^\star })(x)$ according to \cite[Theorem 4.1(a)]{beck2012smoothing} (\cf~\eqref{eq:inf-conv-dual}). Now, since $d\coloneqq h^\star \in \calF_{M_d,\nu_d}$, the result follows from \propref{thm:hmugsc} and \propref{thm:gs-gsc}.
\end{proof}
Under the assumptions of \propref{thm:structured-gsc}, $\Omega_s^C(x;\mu)$ provides a self-concordant smooth approximation of $\Omega(x)$ with $\mathbb{V} \equiv \rr^n$. In this case, $\omega = 0$ in \propref{thm:structured-gsc} and the prox-function $d$ in \eqref{eq:nesterovsmooth} is given by $h^\star $, the dual of $h\in \gsc{h}$.
\subsection{Prox-decomposition and smoothness properties}\label{sec:prox-decomp}
An important property of the function $g=\calR + \Omega^C$ we want to infer here is its prox-decomposition property \cite{yu2013decomposing} in which the (unscaled) proximal operator of $g$ satisfies
\begin{align}
	\prox_{g} = \prox_{\Omega^C} \circ \prox_{\calR}.
\end{align}
Under our assumptions on $g$ and $h$, this property extends for the inf-conv regularization (and hence the self-concordant smoothing framework)\footnote{Additional assumptions may be required to hold in order to accurately define this property in our framework, \eg, nonoverlapping groups in case of the sparse-group lasso problem, in which case, $\mathbb{V}$ is the space $\rr^n$.}. To see this, let $\mathbb{V} \equiv \rr^n$, and note the following equivalent expression for the definition of inf-convolution \eqref{eq:infconv}:
\begin{align*}
	(\infconv{\calR}{h_\mu})(x) = \inf\limits_{\substack{(u,v)\in\rr^n\times\rr^n\\u+v=x}}\left\{\calR(u)+h_\mu(v)\right\}.
\end{align*}
Define also the function $r_s\colon \rr \times \rr \to \rr$ such that
\begin{align*}
	(\infconv{\calR}{h_\mu})(x) \equiv \sum_{i=1}^{n} r_s(x^{(i)};\mu).
\end{align*}

The next result follows, highlighting what we propose as the \emph{inf-decomposition} property.
\begin{proposition}\label{thm:inf-decomposition}
	Let $g \in \pclsc{\rr^n}$ be given as the sum $g(x)=\calR(x) + \Omega^C(x)$. Suppose that the function $h \in \pclsc{\rr^n}$ is supercoercive and define $z \coloneqq [r_s(x^{(1)};\mu),\ldots,r_s(x^{(n)};\mu)]^\top$. Then the regularization process $(\infconv{g}{h_\mu})(x)$, for all $\mu \in \rp$, is given by the composition
	\begin{align}
		(\infconv{g}{h_\mu})(x) = (\infconv{\Omega^C}{h_\mu})(z).
	\end{align}
\end{proposition}
\begin{proof}
	The exactness of the inf-conv regularization process by \propref{thm:exact} allows to infer
	\begin{align*}
		(\infconv{\Omega^C}{h_\mu})(z) &= \inf\limits_{\substack{(u,v)\in\rr^n\times\rr^n\\u+v=z}}\left\{\Omega^C(u)+h_\mu(v)\right\}\\
		&= \inf\limits_{\substack{(u,v)\in\rr^n\times\rr^n\\2u+v=x}}\left\{\calR(u)+\Omega^C(u)+h_\mu(v)\right\}\\
		&= (\infconv{(\calR + \Omega^C)}{h_\mu})(x) = (\infconv{g}{h_\mu})(x).
	\end{align*}
	
\end{proof}

Given the smoothness properties of $\infconv{\Omega^C}{h_\mu}$ and $\infconv{\calR}{h_\mu}$, we can apply the chain rule to obtain the derivatives of their composition $\infconv{g}{h_\mu}$. Precisely, \cite[Lemma 2.1]{sun2006strong} provides sufficient conditions for the validity of the derivatives obtained via the chain rule for composite functions, which are indeed satisfied for $\infconv{g}{h_\mu}$ by our assumptions.
\section{Convergence analysis}\label{sec:convergence}
We analyze the convergence of Algorithms \ref{alg:NewtonSCOREProx} and \ref{alg:GGNSCOREProx} under the proposed smoothing framework. In view of the numerical examples considered in \Secref{sec:experiments}, we restrict our analysis to the case $2\le\nu\le3$. However, similar convergence properties are expected to hold for the general case $\nu\in \rp$, as the key bounds describing generalized self-concordant functions hold similarly for all of these cases (see, \eg, the Section 2 and concluding remark of \cite{sun2019generalized}). We define the following metric term, taking the local norm $\norm{\cdot}_x$ with respect to $g_s$:
\begin{align}\label{eq:d-metric}
	d_\nu(x,y) \coloneqq \begin{cases}
		M_g\norm{y-x} &\text{if } \nu = 2,\\
		\left(\frac{\nu}{2}-1\right)M_g\norm{y-x}_2^{3-\nu}\norm{y-x}_x^{\nu-2} &\text{if }\nu>2.
	\end{cases}
\end{align}
We introduce the notations $H_\star^{g }\equiv \hessn g_s(x^\star )$, $H_\star^{f }\equiv \hessn f(x^\star )$ and $H_\star \equiv\hessn q(x^\star )$. Recall also the notations $H_k^g\equiv \hessn g_s(x_k)$, $H_k^f\equiv \hessn f(x_k)$ and $H_k\equiv\hessn q(x_k)$ at $x_k$. Furthermore, we define the following matrices associated with any given twice differentiable function $f$:
\begin{subequations}
	\begin{align}
		\Sigma_f^{x,y} &\coloneqq \int_0^1 \left(\hessn f(x+\tau(y-x)) - \hessn f(x)\right)d\tau,\\
		\Upsilon_f^{x,y} &\coloneqq \hessn f(x)^{-1/2} \Sigma_f^{x,y} \hessn f(x)^{-1/2}.
	\end{align}
\end{subequations}
We begin by stating some useful preliminary results. The following result provides bounds on the function $g_s$ in \eqref{eq:partialsmooth-prob}.
\begin{lemma}\cite[Proposition 10]{sun2019generalized}\label{thm:g-bound0}
	Suppose that \ref{ass:p3}--\ref{ass:p4} hold. Then, given any $x,y\in \dom g$, we have
	\begin{align}
		\omega_\nu(-d_\nu(x,y))\|y-x\|_x^2 \le g_s(y) - g_s(x) - \langle\gradn g_s(x), y-x\rangle \le \omega_\nu(d_\nu(x,y))\|y-x\|_x^2,
	\end{align}
	in which, if $\nu>2$, the right-hand side inequality holds if $d_\nu(x,y) < 1$, and
	\begin{align}\label{eq:omega-nu}
		\omega_\nu(\tau) \coloneqq \begin{cases}
			\frac{\exp(\tau)-\tau-1}{\tau^2} &\text{if } \nu = 2,\\
			\frac{-\tau - \ln(1-\tau)}{\tau^2} &\text{if } \nu=3,\\
			\frac{(1-\tau)\ln(1-\tau)+\tau}{\tau^2} &\text{if } \nu=4,\\
			\left(\frac{\nu-2}{4-\nu}\right)\frac{1}{\tau}\left[\frac{\nu-2}{2(3-\nu)\tau}\left((1-\tau)\frac{2(3-\nu)}{2-\nu}-1\right)-1\right] &\text{otherwise}.
		\end{cases}
	\end{align}
\end{lemma}
The next two lemmas are instrumental in our convergence analysis, and are immediate consequences of the (local) Hessian regularity of the smooth functions $f$ and $g_s$ in \eqref{eq:partialsmooth-prob}.
\begin{lemma}\cite[Lemma 1.2.4]{nesterov2018lectures}\label{thm:gradf-bound}
	For any given $x,y\in \dom f$, we have
	\begin{align}
		\norm{\gradn f(y) - \gradn f(x) - \hessn f(x)(y-x)} \le \frac{L_f}{2}\norm{y-x}^2,\\
		\abs{f(y) - f(x) - \langle\gradn f(x), y-x\rangle - \frac{1}{2}\langle\hessn f(x)(y-x),y-x\rangle} \le \frac{L_f}{6}\norm{y-x}^3.
	\end{align}
\end{lemma}
\begin{lemma}\cite[Lemma 2]{sun2019generalized}\label{thm:gradg-bound}
	For any given $x,y\in \dom g$, $\Upsilon_{g_s}^{x,y}$ satisfies
	\begin{align*}
		\|\Upsilon_{g_s}^{x,y}\| \le R_\nu(d_\nu(x,y))d_\nu(x,y),
	\end{align*}
	where, for $\tau\in [0,1)$, $R_\nu(\tau)$ is defined by
	\begin{align}\label{eq:r-nu}
		R_\nu(\tau) \coloneqq \begin{cases}
			\left(\frac{3}{2}+\frac{\tau}{3}\right)\exp(\tau) & \text{if }\nu=2,\\
			\frac{1-\left(1-\tau\right)^{\frac{4-\nu}{\nu-2}}-\left(\frac{4-\nu}{\nu-2}\right)\tau\left(1-\tau\right)^{\frac{4-\nu}{\nu-2}}}{\left(\frac{4-\nu}{\nu-2}\right)\tau^2\left(1-\tau\right)^{\frac{4-\nu}{\nu-2}}} & \text{if } \nu\in(2,3].
		\end{cases}
	\end{align}
\end{lemma}
\paragraph*{Global convergence.} We establish a global convergence result for the proximal quasi-Newton scheme \eqref{eq:proxnewton-general}. Specifically, we show that the iterates produced by this scheme decrease the objective function value in \eqref{eq:prob} when the step lengths are chosen according to \eqref{eq:steplength} with $\alpha_k\in(0,1]$. Consequently, global convergence follows.

Let us define the following mapping:
\begin{align}\label{eq:map-G}
	\mapg{x_k} \coloneqq \frac{1}{\bar{\alpha}_k}H_k\left(x_k - \prox_{\alpha_k g}^{H_k^g}(x_k - \bar{\alpha}_kH_k^{-1}\gradn q(x_k))\right).
\end{align}
Clearly, \eqref{eq:proxnewton-general} is equivalent to
\begin{align}\label{eq:equiv-proxnewton-general}
	x_{k+1} = x_k - \bar{\alpha}_kH_k^{-1}\mapg{x_k}.
\end{align} 
Using \eqref{eq:optimality-conditions} with $Q_k = H_k^g$ and the definition of the (scaled) proximal operator, $\mapg{x_k}$ satisfies
\begin{align}\label{eq:mapg-optimality}
	\mapg{x_k} \in \gradn q(x_k) + \partial g(x_k - \bar{\alpha}_kH_k^{-1}\mapg{x_k}).
\end{align}
Moreover, $\mapg{\bar{x}}= 0$ if and only if $\bar{x}$ solves problem \eqref{eq:partialsmooth-prob}.
\begin{proposition}\label{thm:sufficient-decrease}
	\sloppy Suppose that \ref{ass:p1}, \ref{ass:p3} and \ref{ass:p4} hold for \eqref{eq:partialsmooth-prob}. Let $\{x_k\}$ be the sequence generated by scheme \eqref{eq:proxnewton-general} for problem \eqref{eq:partialsmooth-prob} and satisfying $\omega_\nu(d_\nu(x_{k+1},x_k))\le \frac{1}{2}$, where $\omega_\nu$ and $d_\nu$ are respectively defined by \eqref{eq:omega-nu} and \eqref{eq:d-metric}. Define $\varepsilon_k^\mu(y)\coloneqq(L_f/6)\norm{y-x_k}^3$, and let $\bar{\alpha}_k$ be specified by \eqref{eq:steplength} with $\alpha_k\in(0,1]$. Then $\{x_k\}$ satisfies
	\begin{align}
		\calL(x_{k+1}) \le \calL(x_k) - \varepsilon_k^\mu(x_{k+1}).
	\end{align}
\end{proposition}
\begin{proof}
	Letting $y = x_k - \bar{\alpha}_kH_k^{-1}G_{\alpha_kg}(x_k)$ and $x=x_k$ in \lemref{thm:gradf-bound}, where $G_{\alpha_kg}$ is defined by \eqref{eq:map-G}, we have
	\begin{align}\label{eq:f-upper}
		f(x_{k+1}) &\le f(x_k) - \bar{\alpha}_k(H_k^{-1}\gradn f(x_k))^\top\mapg{x_k} + \frac{\bar{\alpha}_k^2}{2}\norm{H_k^{-1}\mapg{x_k}}_{H_k^f}^2\nonumber\\& + \quad \frac{\bar{\alpha}_k^3L_f}{6}\norm{H_k^{-1}\mapg{x_k}}^3.
	\end{align}
	Using $\calL(x_{k+1})\coloneqq f(x_{k+1}) + g(x_{k+1})$ and \eqref{eq:f-upper}, we get
	\begin{align}\label{eq:l-upper-1}
		\calL(x_{k+1}) &\le f(x_k) - \bar{\alpha}_k(H_k^{-1}\gradn f(x_k))^\top\mapg{x_k} + \frac{\bar{\alpha}_k^2}{2}\norm{H_k^{-1}\mapg{x_k}}_{H_k^f}^2\nonumber\\
		&\quad + \frac{\bar{\alpha}_k^3L_f}{6}\norm{H_k^{-1}\mapg{x_k}}^3 + g(x_k - \bar{\alpha}_kH_k^{-1}\mapg{x_k})\nonumber\\
		&\stackrel{\lemref{thm:gradf-bound}}{\le} f(z) - \langle\gradn f(x_k), z-x_k\rangle - \frac{1}{2}\norm{z-x_k}_{H_k^f}^2 + \frac{L_f}{6}\norm{z-x_k}^3\nonumber\\
		&\quad - \bar{\alpha}_k(H_k^{-1}\gradn f(x_k))^\top\mapg{x_k} + \frac{\bar{\alpha}_k^2}{2}\norm{H_k^{-1}\mapg{x_k}}_{H_k^f}^2\nonumber\\
		&\quad + \frac{\bar{\alpha}_k^3L_f}{6}\norm{H_k^{-1}\mapg{x_k}}^3 + g(x_k - \bar{\alpha}_kH_k^{-1}\mapg{x_k}).
	\end{align}
	In the above, we used the lower bound in \lemref{thm:gradf-bound} on $f(z)$. By the convexity of $g$, we have $g(z) - g(x_{x_{k+1}}) \ge v^\top(z-x_{k+1})$ for all $v\in \partial g(x_{k+1})$. Now since from \eqref{eq:mapg-optimality}, we have $\mapg{x_k} - \gradn q(x_k)\in \partial g(x_k - \bar{\alpha}_kH_k^{-1}\mapg{x_k})$, and noting that $\gradn q - \gradn f = \gradn g_s$, \eqref{eq:l-upper-1} gives
	\begin{align}\label{eq:l-upper-2}
		&\calL(x_{k+1}) \le f(z) + g(z) - \langle\gradn f(x_k), z-x_k\rangle - \frac{1}{2}\norm{z-x_k}_{H_k^f}^2 + \frac{L_f}{6}\norm{z-x_k}^3\nonumber\\
		&\quad - \bar{\alpha}_k(H_k^{-1}\gradn f(x_k))^\top\mapg{x_k} + \frac{\bar{\alpha}_k^2}{2}\norm{H_k^{-1}\mapg{x_k}}_{H_k^f}^2 + \frac{\bar{\alpha}_k^3L_f}{6}\norm{H_k^{-1}\mapg{x_k}}^3\nonumber\\
		&\quad - (\mapg{x_k} - \gradn q(x_k))^\top(z - x_k + \bar{\alpha}_kH_k^{-1}\mapg{x_k})\nonumber\\
		&\le \calL(z) - \langle\gradn f(x_k), z-x_k\rangle - \frac{1}{2}\norm{z-x_k}_{H_k^f}^2 - \bar{\alpha}_k(H_k^{-1}\gradn f(x_k))^\top\mapg{x_k}\nonumber\\
		&\quad + \frac{\bar{\alpha}_k^2}{2}\norm{H_k^{-1}\mapg{x_k}}_{H_k^f}^2 + \frac{L_f}{6}\norm{z-x_k}^3 + \frac{\bar{\alpha}_k^3L_f}{6}\norm{H_k^{-1}\mapg{x_k}}^3\nonumber\\
		&\quad - \mapg{x_k}^\top(z - x_k) - \frac{\bar{\alpha}_k^2}{2}\langle H_k^{-1}\mapg{x_k},\mapg{x_k}\rangle\nonumber\\
		&\quad - \gradn q(x_k)^\top(z - x_k + \bar{\alpha}_kH_k^{-1}\mapg{x_k}) \nonumber\\
		&= \calL(z) + \mapg{x_k}^\top(x_k - z) + \frac{\bar{\alpha}_k^2}{2} \langle H_k^{-1}(H_k^fH_k^{-1} - I_n)\mapg{x_k}, \mapg{x_k}\rangle\nonumber\\
		&\quad + \gradn g_s(x_k)^\top(z-x_k) + \bar{\alpha}_k(H_k^{-1}\gradn g_s(x_k))^\top\mapg{x_k} - \frac{1}{2}\norm{z-x_k}_{H_k^f}^2\nonumber\\
		&\quad + \frac{L_f}{6}\norm{z-x_k}^3 + \frac{\bar{\alpha}_k^3L_f}{6}\norm{H_k^{-1}\mapg{x_k}}^3,
	\end{align}
	\sloppy where the second inequality results from the fact that $\langle H_k^{-1}\mapg{x_k},\mapg{x_k}\rangle \in \rnn$ and $\bar{\alpha}_k \ge \bar{\alpha}_k^2$ for $0<\bar{\alpha}_k\le 1$. Now set $z=x_k$ in \eqref{eq:l-upper-2} and use the following relations from \eqref{eq:equiv-proxnewton-general}:
	\begin{align*}
		\bar{\alpha}_kH_k^{-1}\mapg{x_k} = x_k - x_{k+1},\quad \mapg{x_k} = \frac{1}{\bar{\alpha}_k}H_k(x_k - x_{k+1}).
	\end{align*}
	We get
	\begin{align}
		\calL(x_{k+1}) &\le \calL(x_k) + \frac{\bar{\alpha}_k^2}{2} \langle H_k^{-1}(H_k^fH_k^{-1} - I_n)\mapg{x_k}, \mapg{x_k}\rangle\nonumber\\
		&\quad + \bar{\alpha}_k(H_k^{-1}\gradn g_s(x_k))^\top\mapg{x_k} + \frac{\bar{\alpha}_k^3L_f}{6}\norm{H_k^{-1}\mapg{x_k}}^3\nonumber\\
		\quad &= \calL(x_k) - \left[\langle \gradn g_s(x_k),x_{k+1} - x_k\rangle + \frac{1}{2}\langle H_k^g(x_{k+1} - x_k), x_{k+1} - x_k\rangle\nonumber\right.\\&\quad +\left. \frac{L_f}{6}\norm{x_{k+1} - x_k}^3\right].\label{eq:l-b0}
	\end{align}
	Now, let us define the following cubic-regularized upper quadratic model of $g_s$ near $x_k$ (\cf~\cite{nesterov2006cubic}):
	\begin{align*}
		\hat{g}_s(y) \coloneqq g_s(x_k) + \langle \gradn g_s(x_k),y - x_k\rangle + \frac{1}{2}\langle H_k^g(y - x_k), y - x_k\rangle + \frac{L_f}{6}\norm{y - x_k}^3,
	\end{align*}
	for $y\in\rr^n$ and $L_f$ given by \ref{ass:p1}. Then, using \lemref{thm:g-bound0} with $x=x_k$, we have
	\begin{align}
		\gps(y)-\hat{g}_s(y) \le \omega_\nu(d_\nu(y,x_k))\norm{y-x_k}_x^2 - \frac{1}{2}\langle H_k^g(y - x_k), y - x_k\rangle - \frac{L_f}{6}\norm{y - x_k}^3.\label{eq:gs-b0}
	\end{align}
	Next, using \eqref{eq:gs-b0} with $y=x_{k+1}$, \eqref{eq:l-b0} gives
	\begin{align*}
		\calL(x_{k+1}) &\le \calL(x_k) + \gps(x_{k+1}) - \hat{g}_s(x_{k+1})\\
		&\le \calL(x_k) + \left(\omega_\nu(d_\nu(x_{k+1},x_k))-\frac{1}{2}\right)\norm{x_{k+1}-x_k}_x^2 - \frac{L_f}{6}\norm{x_{k+1} - x_k}^3,
	\end{align*}
	which proves the result.
\end{proof}
A straightforward implication of \propref{thm:sufficient-decrease} is that the sequence $\{\calL(x_k)\}$ is monotonically decreasing if $\bar{\delta}_k \coloneqq x_{k+1} - x_k \ne 0$. Consider the set of indices
\begin{align}\label{eq:subsequence-idx}
	\calK_S \coloneqq \set{k \text{ such that } x_k \in S \text{ and } S \text{ is a subsequence of } \{x_k\}}.
\end{align}
Then, for all $k_j\in \calK_S$, $\{x_{k_j}\}$ converges to some $x^\star $.
\begin{lemma}\label{thm:dist-zero}
	Let an iterate $x_k$ be generated by the scheme \eqref{eq:proxnewton-general} for problem \eqref{eq:partialsmooth-prob}. Then, $x_k$ is a stationary point of $\calL$ if and only if $\bar{\delta}_k=0$.
\end{lemma}
\begin{proof}
	The statement holds true by our characterization of the optimiality conditions in \eqref{eq:optimality-conditions} with $Q_k = H_k^g$.
\end{proof}
\begin{theorem}\label{thm:stationarity}
	Let $\{x_k\} \subset \rr^n$ in \propref{thm:sufficient-decrease}. Then every limit point $x^\star $ of $\{x_k\}$ at which \eqref{eq:optimality-conditions} holds with $Q_k = H_k^g$ is a stationary point of the objective function $\calL$ in problem \eqref{eq:prob}.
\end{theorem}
\begin{proof}
	\propref{thm:sufficient-decrease} implies $\{\calL(x_k)\}$ is non-increasing and bounded below. Hence, it converges to a finite value $\calL^\star $. Consequently (and from the proof of \propref{thm:epi-minimizer}), the sequence of iterates $\{x_k\}$ generated from \eqref{eq:proxnewton-general} is bounded, and every limit point exists. Let $x^\star $ be a limit point of $\{x_k\}$, and now consider all $k_j \in \calK_S$ with $\{x_{k_j}\} \to x^\star $, where $\calK_S$ is defined by \eqref{eq:subsequence-idx}. The relation in \eqref{eq:grad-consistency} implies inclusion in both directions, and hence since $g_s \epiarrow g$, if $\{x_{k_j}\}$ is such that
	\begin{align}
		\limsup_{\substack{x_{k_j}\to x^\star \\\mu\downarrow 0}} \gradn g_s(x_{k_j};\mu) \to 0,\label{eq:gs-stationarity}
	\end{align}
	one finds $x^\star $ is a stationary point of $g$ \cite{burke2013epi}. For any suitably chosen fixed $\mu \in \rp$, it suffices that both properties \eqref{eq:grad-consistency} and \eqref{eq:gs-stationarity} hold only approximately with respect to \propref{thm:epi-minimizer} as they pertain only to the smooth part of the problem. Taking the limit of \eqref{eq:optimality-conditions} as $k_j\to \infty$ with $Q_k = H_k^g$, the result follows from \lemref{thm:dist-zero}. Precisely, $\bar{\delta}_{k_j} \to 0$, and hence all the limit points of $\{x_k\}$ are stationary points of $\calL$.
\end{proof}

\paragraph*{How to choose $\alpha_k$.}
In previous results, we did not specify a particular way to choose $\alpha_k$. Our algorithms converge for any value of $\alpha_k\in(0,1]$. Compared to the step length selection rule proposed in \cite{sun2019generalized}, for instance, our approach and analysis do not directly rely on the actual value of $\nu$ in the choice of both $\bar{\alpha}_k$ and $\alpha_k$. Indeed, in the context of minimizing a function $g_s\in\gsc{g}$, an optimal choice for $\bar{\alpha}_k$, in view of \cite{sun2019generalized}, corresponds to setting
\begin{align*}
	\alpha_k = \begin{cases}
		\frac{\ln(1+d_k)(1 + M\eta_k)}{d_k} &\text{if } \nu=2,\\
		\frac{2(1+M_g\eta_k)}{2+M_g\eta_k} &\text{if } \nu=3,
	\end{cases}
\end{align*}
where $d_k\coloneqq M_g\|\hessn H_k^{g^{-1}}\gradn g_s(x_k)\|$ and in each case, it can be shown that $\bar{\alpha}_k\in(0,1)$. However, choosing $\alpha_k$ this way does not guarantee certain theoretical bounds in the context of the framework studied in this work, especially for $\nu=2$. We therefore propose to leave $\alpha_k$ as a hyperparameter that must satisfy $0<\alpha_k\equiv\alpha\le1$. This however provides the freedom to exploit specific properties about the function $f$, when they are known to hold. One of such properties is the global Lipschitz continuity of $\gradn f$, where supposing the Lipschitz constant $L$ is known, one may set
\begin{align*}
	\alpha_k = \min\{1/L,1\}.
\end{align*}
\paragraph*{Local convergence.}
We next discuss the local convergence properties of Algorithms \ref{alg:NewtonSCOREProx} and \ref{alg:GGNSCOREProx}. In our discussion, we take the local norm $\norm{\cdot}_x$ (and its dual) with respect to $g_s$, and the standard Euclidean norm $\norm{\cdot}$ with respect to the (local) Euclidean ball $\mathcal{B}_{r_0}(\cdot) \subset \calE_r(\cdot)$. We also remark that, by definition, $\omega_\nu$ is a strictly increasing function.
\begin{theorem}\label{thm:alg1}
	Suppose that \ref{ass:p1}--\ref{ass:p4} hold, and let $x^\star$ be an optimal solution of \eqref{eq:partialsmooth-prob}. Let $\{x_k\}$ be the sequence of iterates generated by Algorithm \ref{alg:NewtonSCOREProx} and define $\lambda_k\coloneqq 1+M_g\omega_\nu(-d_\nu(x^\star,x_k))\norm{x_k-x^\star}_{x_k}$, where $\omega_\nu$ is defined by \eqref{eq:omega-nu}. Then starting from a point $x_0 \in \calE_r(x^\star)$, if $d_\nu(x^\star,x_k)<1$ with $d_\nu$ defined by \eqref{eq:d-metric}, the sequence $\{x_k\}$ satisfies
	\begin{align}\label{eq:conv-rate-xk-alg1}
		\norm{x_{k+1}-x^\star}_{x^\star} \le \vartheta_k\norm{x_k-x^\star} + R_k\norm{x_k-x^\star}_{x^\star} + \frac{L_f}{2\sqrt{\rho_2}}\norm{x_k-x^\star}^2,
	\end{align}
	where $\vartheta_k \coloneqq (L_1+L_2)(\lambda_k-\alpha_k)/(\lambda_k\sqrt{\rho})$, $\alpha_k\in(0,1]$, $R_k\coloneqq R_\nu(d_\nu(x^\star,x_k))d_\nu(x^\star,x_k)$ with $R_\nu$ defined by \eqref{eq:r-nu}.
\end{theorem}
\begin{proof}
	The iterative process of Algorithm \ref{alg:NewtonSCOREProx} is given by
	\begin{align*}
		x_{k+1} = \prox_{\alpha_k g}^{H_k^g} (x_k - \bar{\alpha}_k\hessn q(x_k)^{-1}\gradn q(x_k)).
	\end{align*}
	In terms of $E_{\bar{x}}$ and $\xi_{\bar{x}}(Q_k,\cdot)$ with $Q_k\equiv H_k^g$, and using the definition of $q$, we have
	\begin{align}
		&\norm{x_{k+1}-x^\star }_{x^\star } = \norm{\prox_{\alpha_k g}^{H_\star^{g }}(E_{x^\star }(x_k) + \xi_{x^\star }(Q_k,x_{k+1})) - \prox_{\alpha_k g}^{H_\star^{g }}(E_{x^\star }(x^\star ))}_{x^\star }\nonumber\\
		&\stackrel{\eqref{eq:nonexpansive}}{\le} \norm{E_{x^\star }(x_k) - E_{x^\star }(x^\star ) + \xi_{x^\star }(Q_k,x_{k+1})}_{x^\star }^{\diamond}\nonumber\\
		&= \norm{H_\star x_k - \bar{\alpha}_k \gradn q(x_k) - H_\star x^\star  + \bar{\alpha}_k q(x^\star )}_{x^\star }^{\diamond}\nonumber\\
		&= \norm{\gradn q(x^\star ) - \gradn q(x_k) + (1-\bar{\alpha}_k)(\gradn q(x_k) - \gradn q(x^\star )) + H_\star (x_k - x^\star )}_{x^\star }^{\diamond}\nonumber\\
		&\le  \norm{\gradn q(x_k) - \gradn q(x^\star ) - H_\star (x_k - x^\star )}_{x^\star }^{\diamond} + (1-\bar{\alpha}_k)\norm{\gradn q(x_k) - \gradn q(x^\star )}_{x^\star }^{\diamond}\nonumber\\
		&\le  \norm{\gradn f(x_k) - \gradn f(x^\star ) - H_\star^{f }(x_k - x^\star )}_{x^\star }^{\diamond} + \norm{\gradn g_s(x_k) - \gradn g_s(x^\star ) - H_\star^{g }(x_k - x^\star )}_{x^\star }^{\diamond} \nonumber\\&\quad + (1-\bar{\alpha}_k)\left(\norm{\gradn f(x_k) - \gradn f(x^\star )}_{x^\star }^{\diamond} + \norm{\gradn g_s(x_k) - \gradn g_s(x^\star )}_{x^\star }^{\diamond}\right). \label{eq:xk-bound}
	\end{align}
	To estimate $\|\gradn f(x_k) - \gradn f(x^\star ) - H_\star^{f }(x_k - x^\star )\|_{x^\star }^{\diamond}$, we note that for $v\in\rr^n$, $\|v\|_{x^\star }^{\diamond} \equiv \|H_k^{g^{\star -\frac{1}{2}}}v\|$ since we take the dual norm with respect to $g_s$. Now, using \ref{ass:p2}, we get that the matrix $H_\star^{g }$ is positive definite and
	\begin{align}\label{eq:hg-bound}
		\|H_k^{g^{\star -\frac{1}{2}}}\| \le \frac{1}{\sqrt{\rho_2}}.
	\end{align}
	Consequently, we have
	\begin{align*}
		\norm{\gradn f(x_k) - \gradn f(x^\star ) - H_\star^{f }(x_k - x^\star )}_{x^\star }^{\diamond} &= \norm{H_k^{g^{\star {-\frac{1}{2}}}}\left(\gradn f(x_k) - \gradn f(x^\star ) - H_\star^{f }(x_k - x^\star )\right)}\\
		&\le \|H_k^{g^{\star {-\frac{1}{2}}}}\|\norm{\gradn f(x_k) - \gradn f(x^\star ) - H_\star^{f }(x_k - x^\star )}\\
		&\stackrel{\lemref{thm:gradf-bound}}{\le} \frac{L_f\|x_k - x^\star \|^2}{2\sqrt{\rho_2}}.
	\end{align*}
	To estimate $\|\gradn g_s(x_k) - \gradn g_s(x^\star ) - H_\star^{g }(x_k - x^\star )\|_{x^\star }^{\diamond}$, we can apply \lemref{thm:gradg-bound} as in the proof of \cite[Theorem 5]{sun2019generalized}, and get
	\begin{align*}
		\norm{\gradn g_s(x_k) - \gradn g_s(x^\star ) - H_\star^{g }(x_k - x^\star )}_{x^\star }^{\diamond} \le R_\nu(d_\nu(x^\star ,x_k))d_\nu(x^\star ,x_k)\norm{x_k-x^\star }_{x^\star }.
	\end{align*}
	Following \cite[p. 195]{sun2019generalized}, we can derive the following inequality in a neighbourhood of the sublevel set of $\calL_s$ in \eqref{eq:partialsmooth-prob} using \lemref{thm:g-bound0} and the convexity of $g_s$:
	\begin{align}
		\|\gradn g_s(x_k)\|_{x_k}^{\diamond} \ge \omega_\nu(-d_\nu(x^\star ,x_k))\|x_k-x^\star \|_{x_k}.\label{eq:hg-positive}
	\end{align}
	In this regard, \eqref{eq:steplength} gives
	\begin{align}
		1-\bar{\alpha}_k \le \frac{\lambda_k - \alpha_k}{\lambda_k}.
	\end{align}
	Next, by \ref{ass:p2}, we deduce
	\begin{align*}
		\norm{\gradn g_s(x_k) - \gradn g_s(x^\star )} \le L_2\norm{x_k-x^\star },
	\end{align*}
	and
	\begin{align*}
		\norm{\gradn f(x_k) - \gradn f(x^\star )} \le L_1\norm{x_k-x^\star }.
	\end{align*}
	Then, using \eqref{eq:hg-bound}, we get
	\begin{align*}
		\norm{\gradn g_s(x_k) - \gradn g_s(x^\star )}_{x^\star }^{\diamond} &= \norm{H_k^{g^{\star -\frac{1}{2}}}\left(\gradn g_s(x_k) - \gradn g_s(x^\star )\right)}\\
		&\le \frac{L_2}{\sqrt{\rho_2}}\norm{x_k-x^\star }.
	\end{align*}
	Similarly,
	\begin{align*}
		\norm{\gradn f(x_k) - \gradn f(x^\star )}_{x^\star }^{\diamond} \le \frac{L_1}{\sqrt{\rho_2}}\norm{x_k-x^\star }.
	\end{align*}
	Finally, putting the above estimates into \eqref{eq:xk-bound}, we obtain \eqref{eq:conv-rate-xk-alg1}.
\end{proof}
To prove the local convergence of Algorithm \ref{alg:GGNSCOREProx}, we need an additional assumption about the behaviour of the Jacobian matrix $J_k$ near $x^\star $. As before, $J_k$ denotes the Jacobian matrix evaluated at $x_k$; likewise, $V_k$ and $e_k$. At $x^\star $, we respectively write $J^\star $, $V^\star $ and $u^\star $. We assume the following:
\begin{enumerate}[label=\enumlabel{G}, ref=\enumref{G}]
	\item $\|J_kv\|\ge \beta_1\|v\|$, $\beta_1 \in \rp$, for all $x_k$ near $x^\star $, and for any $v\in \rr^n$.\label{ass:g1}
\end{enumerate}
For $f$ defined by \eqref{eq:f-ggn}, condition \ref{ass:g1} implies that the singular values of $J_k$ are uniformly bounded away from zero, at least locally. Let the unaugmented version of the residual vector $e_k$ be denoted by $\tilde{e}_k$, that is,
\begin{align*} \tilde{e}_k\coloneqq[l'_{\hat{y}^{(1)}}(y^{(1)},\hat{y}^{(1)}),\ldots,l'_{\hat{y}^{(m)}}(y^{(m)},\hat{y}^{(m)})]^\top\in\rr^m.
\end{align*}
Define the following matrix:
\begin{align}
	W_k^\top  &\coloneqq \begin{bmatrix}
		\hess{\hat{y}{^{(1)}}}(x^{(1)}) & \hess{\hat{y}{^{(2)}}}(x^{(1)}) & \cdots & \hess{\hat{y}{^{(m)}}}(x^{(1)})\\
		\hess{\hat{y}{^{(1)}}}(x^{(2)}) & \hess{\hat{y}{^{(2)}}}(x^{(2)}) & \cdots & \hess{\hat{y}{^{(m)}}}(x^{(2)})\\
		\vdots&\vdots&&\vdots\\
		\hess{\hat{y}{^{(1)}}}(x^{(n)}) & \hess{\hat{y}{^{(2)}}}(x^{(n)}) & \cdots & \hess{\hat{y}{^{(m)}}}(x^{(n)})
	\end{bmatrix} \in \rr^{n\times m}. \label{eq:Junaug}
\end{align}
We note that the ``full'' Hessian matrix $H_k$ can be expressed as
\begin{align}
	H_k \equiv J_k^\top V_k J_k + (\mathbf{1}\otimes(W_k^\top\tilde{e}_k))^\top + H_k^g,\label{eq:hk-hat}
\end{align}
where $\mathbf{1}\in\rr^{n\times 1}$ is the $n \times 1$ matrix of ones and $\otimes$ denotes the outer product. By \ref{ass:p1}, \ref{ass:p2} and the Lipschitz continuity of $g_s$ around $x^\star $ in \propref{thm:g-lip}, we have: for $r$ small enough, there exists a constant $\beta_2\in \rp$ such that $\norm{\tilde{e}_k}\le \beta_2$ near $x^\star $. Furthermore by our assumptions (see, \eg, \cite[Theorem 10.1]{nocedal1999numerical}), we deduce that there exists $\beta_3\in \rp$ such that $\|W_k\|\le \beta_3$ near $x^\star $.

The next result follows. Note that for Algorithm \ref{alg:GGNSCOREProx}, we consider the case where $f$ in problem \eqref{eq:partialsmooth-prob} may, in general, be expressed in the form \eqref{eq:f-ggn}.
\begin{theorem}\label{thm:alg2}
	Suppose that \ref{ass:p1}--\ref{ass:p4} hold, and let $x^\star$ be an optimal solution of \eqref{eq:partialsmooth-prob} where $f$ is defined by \eqref{eq:f-ggn}. Additionally, let \ref{ass:g1} hold for the Jacobian matrix $J_k$ defined by \eqref{eq:Jaug}. Let $\{x_k\}$ be the sequence of iterates generated by Algorithm \ref{alg:GGNSCOREProx}, and define $\lambda_k\coloneqq 1+M_g\omega_\nu(-d_\nu(x^\star,x_k))\norm{x_k-x^\star}_{x_k}$, where $\omega_\nu$ is defined by \eqref{eq:omega-nu}. Then starting from a point $x_0 \in \calE_r(x^\star)$, if $d_\nu(x^\star,x_k)<1$ with $d_\nu$ defined by \eqref{eq:d-metric}, the sequence $\{x_k\}$ satisfies
	\begin{align}\label{eq:conv-rate-xk-alg2}
		\norm{x_{k+1}-x^\star}_{x^\star} \le \vartheta_k\norm{x_k-x^\star} + R_k\norm{x_k-x^\star}_{x^\star} + \frac{L_f}{2\sqrt{\rho_2}}\norm{x_k-x^\star}^2,
	\end{align}
	where $R_k$ is as defined in \thmref{thm:alg1}, $\vartheta_k \coloneqq (\lambda_k(L_1+L_2)(\lambda_k-\alpha_k)+\tilde{\beta})/\sqrt{\rho_2}$, $\alpha_k\in(0,1]$, and $\tilde{\beta}\coloneqq\beta_2\beta_3\in \rp$.
\end{theorem}
\begin{proof}
	Let $\hat{H}_k \coloneqq J_k^\top V_k J_k + H_k^g$, and consider the iterative process of Algorithm \ref{alg:GGNSCOREProx} given by
	\begin{align*}
		x_{k+1} = \prox_{\alpha_k g}^{H_k^g} (x_k - \bar{\alpha}_k \hat{H}_k^{-1}J_k^\top e_k).
	\end{align*}
	We first note that $J_k^\top e_k$ is a compact way of writing $\gradn f(x_k) + \gradn g_s(x_k) \eqqcolon \gradn q(x_k)$, where $f$ is given by \eqref{eq:f-ggn}. Following the proof of \thmref{thm:alg1}, we have
	\begin{align}
		&\norm{x_{k+1}-x^\star }_{x^\star } = \norm{\prox_{\alpha_k g}^{H_\star^{g }}(E_{x^\star }(x_k) + \xi_{x^\star }(Q_k,x_{k+1})) - \prox_{\alpha_k g}^{H_\star^{g }}(E_{x^\star }(x^\star ))}_{x^\star }\nonumber\\
		&\le  \norm{\gradn q(x_k) - \gradn q(x^\star ) - \hat{H}_k^\star (x_k - x^\star )}_{x^\star }^{\diamond} + (1-\bar{\alpha}_k)\norm{\gradn q(x_k) - \gradn q(x^\star )}_{x^\star }^{\diamond}.\label{eq:xk-bound-2}
	\end{align}
	Let $W^\star $ and $\tilde{u}^\star $ respectively denote expressions for $W_k$ and $\tilde{u}$ evaluated at $x^\star $. Substituting \eqref{eq:hk-hat} into \eqref{eq:xk-bound-2} and using \eqref{eq:hg-bound} in the estimate
	\begin{align*}
		\norm{(\mathbf{1}\otimes(W^{\star ^\top}\tilde{e}_k))^\top(x_k-x^\star )}_{x^\star }^{\diamond}\le \norm{H_k^{g^{\star -\frac{1}{2}}}(\mathbf{1}\otimes(W^{\star ^\top}\tilde{u}^\star ))^\top}\norm{x_k-x^\star },
	\end{align*}
	where $W_k$ is defined by \eqref{eq:Junaug}, we get
	\begin{align}
		&\norm{x_{k+1}-x^\star }_{x^\star } \le  \norm{\gradn q(x_k) - \gradn q(x^\star ) - H_\star (x_k - x^\star )}_{x^\star }^{\diamond} + \norm{(\mathbf{1}\otimes(W^{\star ^\top}\tilde{u}^\star ))^\top(x_k-x^\star )}_{x^\star }^{\diamond}\nonumber\\
		&\quad + (1-\bar{\alpha}_k)\norm{\gradn q(x_k) - \gradn q(x^\star )}_{x^\star }^{\diamond}\nonumber\\
		&\le  \norm{\gradn f(x_k) - \gradn f(x^\star ) - H_\star^{f }(x_k - x^\star )}_{x^\star }^{\diamond} + \norm{\gradn g_s(x_k) - \gradn g_s(x^\star ) - H_\star^{g }(x_k - x^\star )}_{x^\star }^{\diamond} \nonumber\\&\quad + (1-\bar{\alpha}_k)\left(\norm{\gradn f(x_k) - \gradn f(x^\star )}_{x^\star }^{\diamond} + \norm{\gradn g_s(x_k) - \gradn g_s(x^\star )}_{x^\star }^{\diamond}\right) + \frac{\tilde{\beta}\norm{x_k - x^\star }}{\sqrt{\rho_2}}, \label{eq:xk-bound-3}
	\end{align}
	where $\tilde{\beta}=\beta_2\beta_3$. Now, using the estimates derived in the proof of \thmref{thm:alg1} in \eqref{eq:xk-bound-3} above, we obtain \eqref{eq:conv-rate-xk-alg2}.
\end{proof}

\section{Numerical experiments}\label{sec:experiments}
In this section, we validate the efficiency of the technique introduced in this paper in numerical examples using both synthetic and real datasets from the LIBSVM repository \cite{chang2011libsvm}. The approach and algorithms proposed in this paper are implemented in the Julia programming language and are available online as an open-source package\footnote{\url{https://github.com/adeyemiadeoye/SelfConcordantSmoothOptimization.jl}. Code to reproduce most of the experiments in this paper can be found in the \textbf{v0.1.0} release.}. We test the performance of Algorithms \ref{alg:NewtonSCOREProx} and \ref{alg:GGNSCOREProx} for various fixed values of $\alpha_k \equiv \alpha \in (0,1]$ (see \figurename~\ref{fig:alpha-plots}). In the remaining parts, we fix $\alpha_k = 1$ and compare our approach with \texttt{PANOC} \cite{stella2017simple}, \texttt{ZeroFPR} \cite{themelis2018forward}, \texttt{OWL-QN} \cite{andrew2007scalable}, proximal gradient \cite{lions1979splitting}, and fast proximal gradient \cite{beck2009fast} algorithms\footnote{We use the open-source package \texttt{ProximalAlgorithms.jl} for the \texttt{PANOC}, \texttt{ZeroFPR}, and fast proximal gradient algorithms, while we use our own implementation of the \texttt{OWL-QN} (modification of \url{https://gist.github.com/yegortk/ce18975200e7dffd1759125972cd54f4}) and proximal gradient methods.}. In the sparse-group lasso experiments, we also compare with the block coordinate descent (\texttt{BCD})\footnote{We use the \texttt{BCD} method of \cite{ndiaye2017gap} which is efficiently implemented with a gap safe screening rule. The open-source implementation can be found in \url{https://github.com/EugeneNdiaye/Gap_Safe_Rules}.} algorithm, and the semismooth Newton augmented Lagrangian (\texttt{SSNAL}) method \cite{li2018highly} which was extended\footnote{We use the freely available implementation provided by the authors in \url{https://github.com/YangjingZhang/SparseGroupLasso}.} in \cite{zhang2020efficient} to solve sparse-group lasso problems. \texttt{BCD} is known to be an efficient algorithm for general regularized problems \cite{friedman2010regularization}, and is used as a standard approach for the sparse-group lasso problem \cite{ida2019fast,friedman2010note,simon2013sparse}. Since the problems considered in our experiments use the $\ell_1$ and $\ell_2$ regularizers, we use $\phi(t) = \frac{1}{p}\sqrt{1+p^2\abs{t}^2}-1$ from \exref{ex:infconv1}, with $p=1$ and derive $g_s$ in problem \eqref{eq:partialsmooth-prob} accordingly.

For a diagonal matrix $H_k^g\in\rr^{n\times n}$, the scaled proximal operator for the $1$- and $2$-norms are obtained using the proximal calculus derived in \cite{becker2019quasi}. Let\ $\hat{d}_k\in\rr^n$ be the vector containing the diagonal entries of $H_k^g$, and let $\beta \in \rp$; the components of $\prox_{\beta \|\cdot\|_1}^{H_k^g}$ and $\prox_{\beta \|\cdot\|}^{H_k^g}$ at iteration $k$ are given, respectively, by:
\begin{enumerate}
	\item $\left(\prox_{\beta \|\cdot\|_1}^{H_k^g}(p_k)\right)^{(i)} = \sign(p_k^{(i)})\max\{|p_k^{(i)}| - \beta \hat{d}_k^{(i)},0\}$, and \quad
	\item $\left(\prox_{\beta \|\cdot\|}^{H_k^g}(p_k)\right)^{(i)} = p_k^{(i)}\max\{1 - \beta\hat{d}_k^{(i)}/\|p_k\|,0\}$.
\end{enumerate}
We terminate each of the tested algorithms either with the default stopping criterion or when $\frac{\|x_k  - x_{k-1}\|}{\max\{\|x_{k-1}\|,1\}} < \eps_{tol}$ with $\eps_{tol} \in \{10^{-6}, 10^{-10}\}$.

All experiments are performed on a laptop with dual (2.30GHz + 2.30GHz) Intel Core i7-11800 H CPU
and 32GB RAM.
\begin{table}[t!]
	\centering
	\small
	\caption{Summary of the real datasets used for sparse logistic regression.}
	\begingroup
	\setlength{\tabcolsep}{10pt}
	\renewcommand{\arraystretch}{1.2}
	\begin{tabular}{cccc}
		\toprule
		Data & $m$ & $n$ & Density \\
		\midrule
		\texttt{mushrooms} & $8124$ & $112$ & $0.19$ \\
		\texttt{phishing} & $11055$ & $68$ & $0.44$ \\
		\texttt{w1a} & $2477$ & $300$ & $0.04$ \\
		\texttt{w2a} & $3470$ & $300$ & $0.04$ \\
		\texttt{w3a} & $4912$ & $300$ & $0.04$ \\
		\texttt{w4a} & $7366$ & $300$ & $0.04$ \\
		\texttt{w5a} & $9888$ & $300$ & $0.04$ \\
		\texttt{w8a} & $49749$ & $300$ & $0.04$ \\
		\texttt{a1a} & $1605$ & $123$ & $0.11$ \\
		\texttt{a2a} & $2265$ & $123$ & $0.11$ \\
		\texttt{a3a} & $3185$ & $123$ & $0.11$ \\
		\texttt{a4a} & $4781$ & $123$ & $0.11$ \\
		\texttt{a5a} & $6414$ & $123$ & $0.11$ \\
		\bottomrule
	\end{tabular}
	\label{tab:data-summary}
	\endgroup
\end{table}
\begin{figure}[t!]
	\centering
	\subfloat{%
		\resizebox*{7.0cm}{!}{\includegraphics{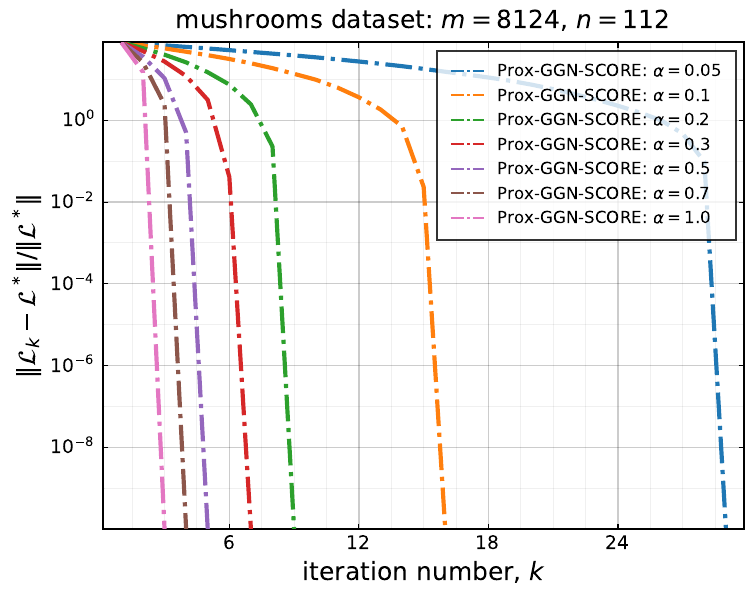}}}\hspace{5pt}
	\subfloat{%
		\resizebox*{7.0cm}{!}{\includegraphics{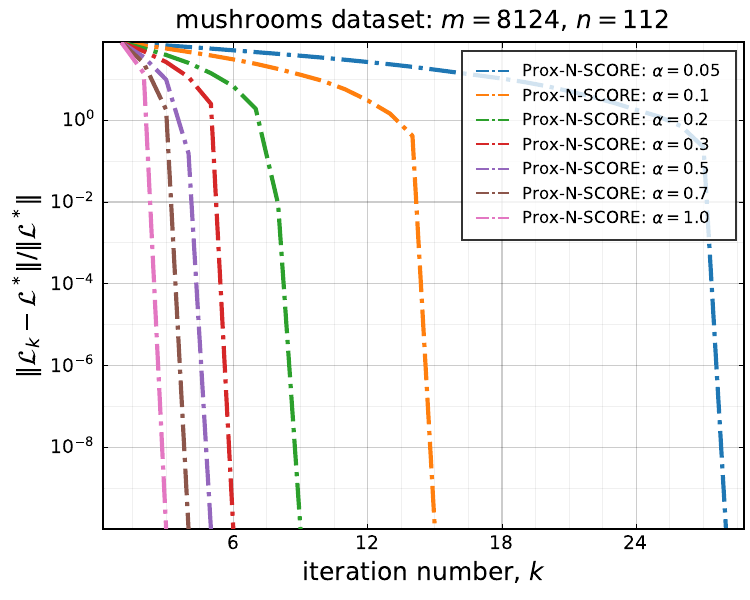}}}
	\vfil
	\subfloat{%
		\resizebox*{7.0cm}{!}{\includegraphics{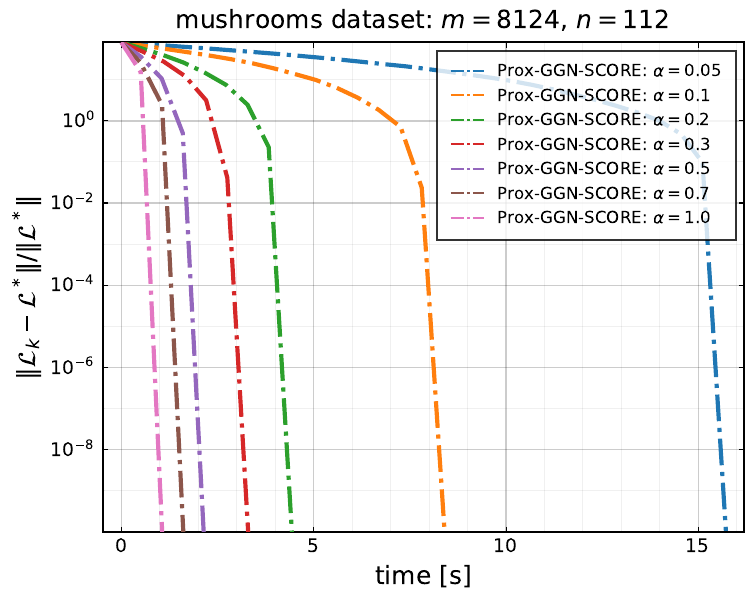}}}\hspace{5pt}
	\subfloat{%
		\resizebox*{7.0cm}{!}{\includegraphics{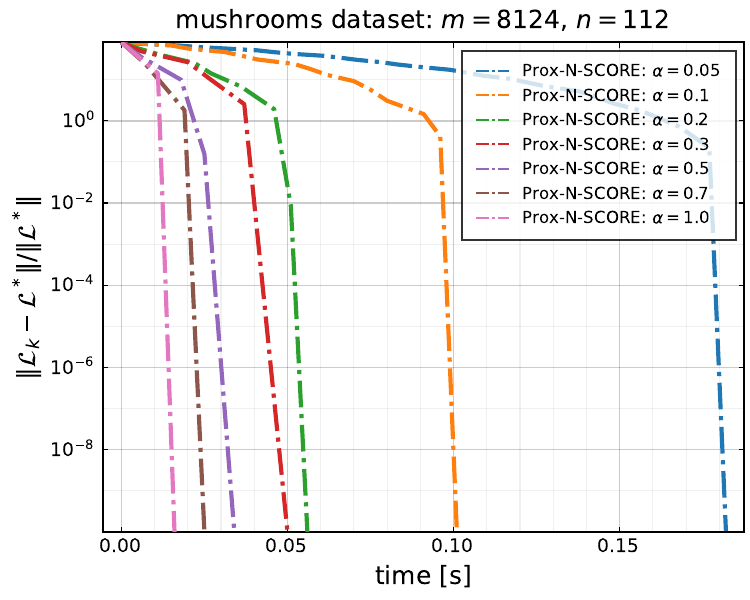}}}
	\caption{Behaviour of \texttt{Prox-N-SCORE} and \texttt{Prox-GGN-SCORE} for different fixed values of $\alpha_k$ in problem \eqref{eq:logexample}.} \label{fig:alpha-plots}
\end{figure}
\begin{figure}[t!]
	\centering
	\subfloat{%
		\resizebox*{7.0cm}{!}{\includegraphics{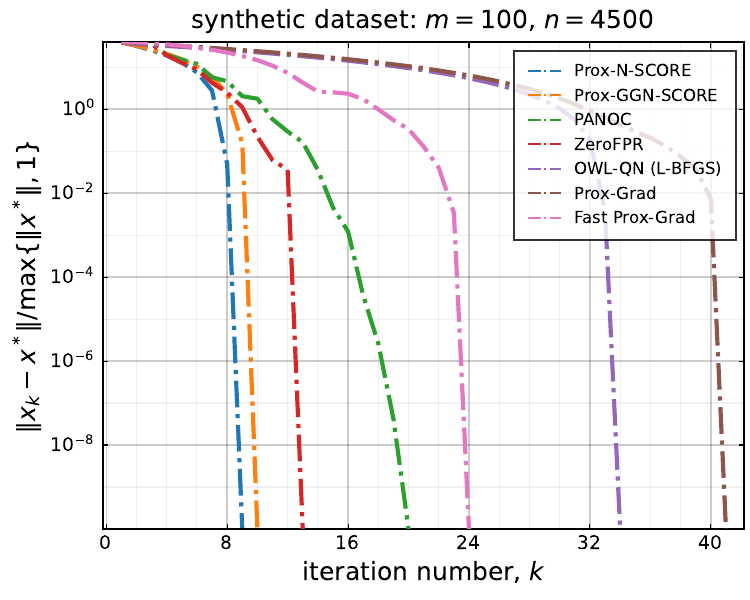}}}\hspace{5pt}
	\subfloat{%
		\resizebox*{7.0cm}{!}{\includegraphics{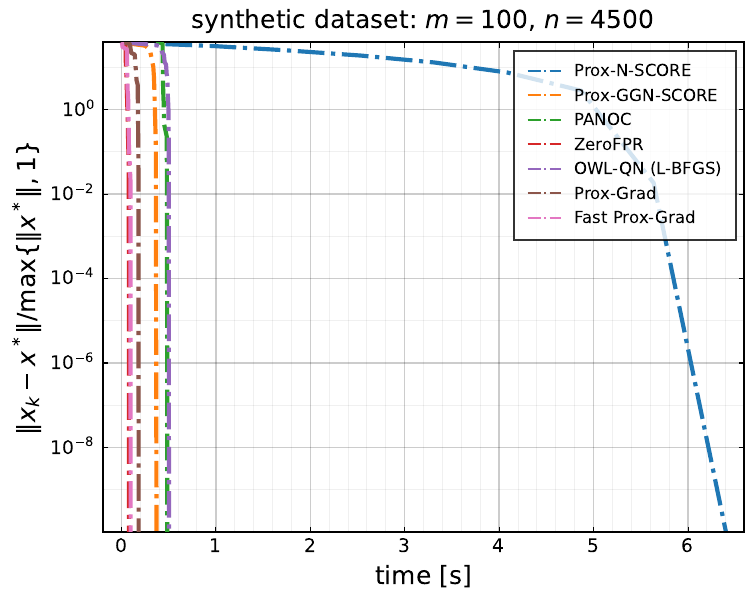}}}
	\vfil
	\subfloat{%
		\resizebox*{7.0cm}{!}{\includegraphics{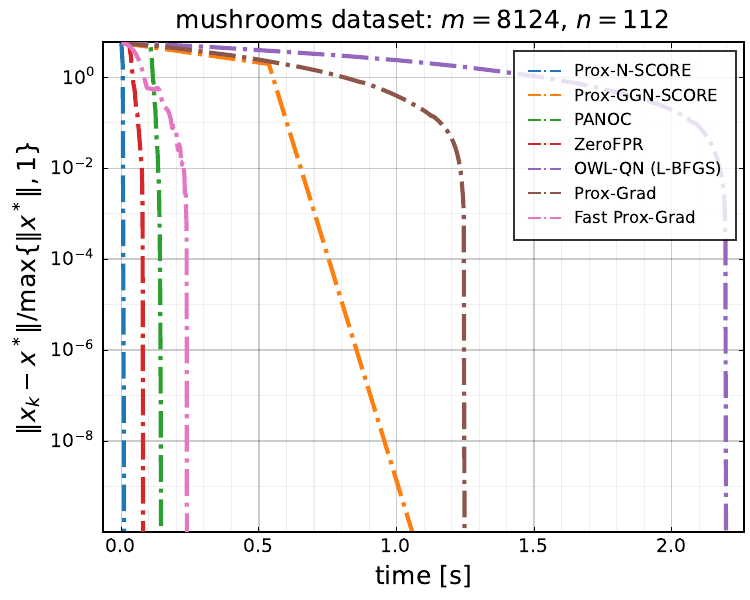}}}\hspace{5pt}
	\subfloat{%
		\resizebox*{7.0cm}{!}{\includegraphics{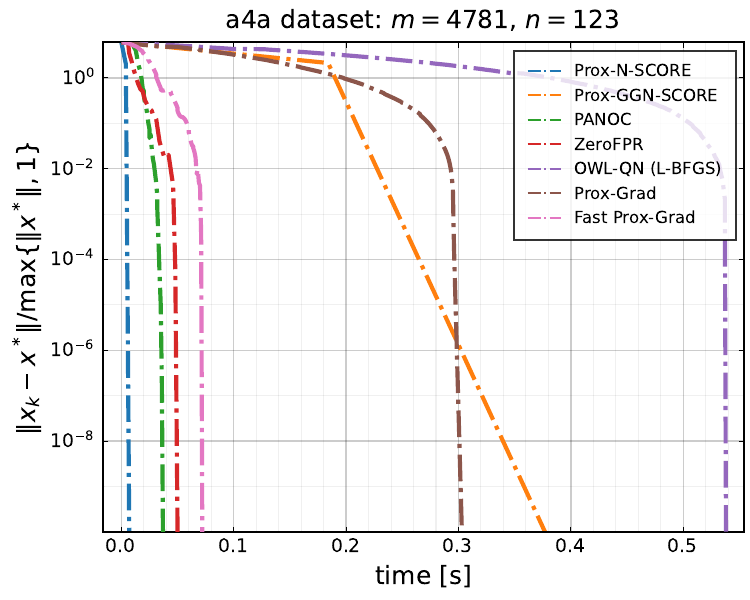}}}
	\caption{Overparameterized problem (first row) and non-overparameterized problems (second row) in \eqref{eq:logexample}. \texttt{Prox-GGN-SCORE} reduces most of the computational burden of \texttt{Prox-N-SCORE} if $m+n_y < n$ (or $m\ll n$). However, \texttt{Prox-N-SCORE} solves the problem faster, and is more stable, if $n < m+n_y$ (or $n\ll m$).} \label{fig:splogl1-loss}
\end{figure}
\begin{figure}[t!]
	\centering
	\includegraphics[scale=0.75]{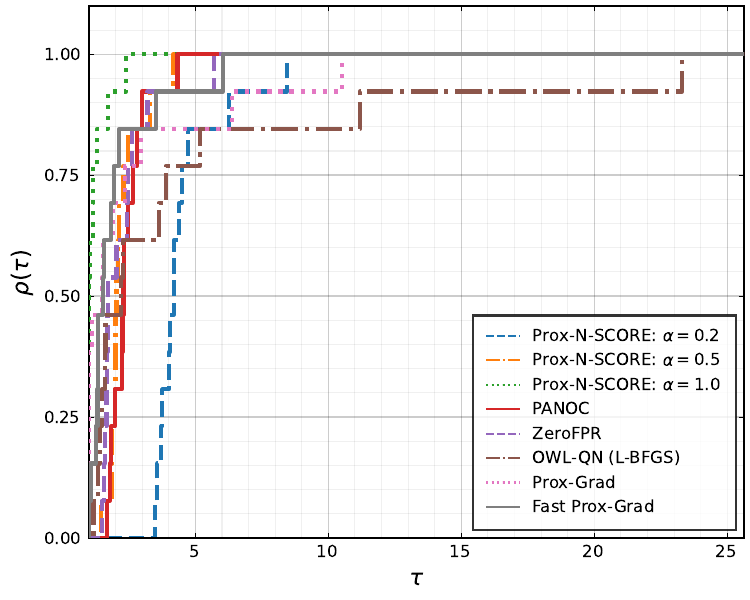}
	\caption{Performance profile (CPU time) for the sparse logistic regression problem \eqref{eq:logexample} using the LIBSVM datasets summarized in \tablename~\ref{tab:data-summary}. Here, $\tau$ denotes the performance ratio (CPU times in seconds) averaged over 20 independent runs with different random initializations, and $\rho(\tau)$ is the corresponding frequency.}
	\label{fig:perf-prof}
\end{figure}
\subsection{Sparse logistic regression}\label{ss:logexample}
We consider the problem of finding a sparse solution $x$ to the following logistic regression problem
\begin{align}\label{eq:logexample}
	\min\limits_{x\in\rr^n} \calL(x) \coloneqq \underbrace{\sum_{i=1}^{m} \log\left(1 + \exp(-y^{(i)}\langle a^{(i)}, x\rangle)\right)}_{\eqqcolon f(x)} + \beta \|x\|_1,
\end{align}
where, in view of \eqref{eq:prob}, $g(x) \coloneqq \beta \|x\|_1$, $\beta\in\rp$, and $a^{(i)} \in \rr^n, y^{(i)} \in \{-1,1\}$ form the data. We perform experiments on both randomly generated data and real datasets summarized in \tablename~\ref{tab:data-summary}. For the synthetic data, we set $\beta = 0.2$, while for the real datasets, we set $\beta = 1$. We fix $\mu =1$ in both Algorithms \ref{alg:NewtonSCOREProx} and \ref{alg:GGNSCOREProx}, and set $\alpha_k = 1/L$ for the proximal gradient algorithm, where $L$ is estimated as $L = \lambda_{max} (A^\top A)$, the columns of $A\in \rr^{n\times m}$ are the vectors $a^{(i)}$ and $\lambda_{max}$ denotes the largest eigenvalue. For the sake of fairness, we provide this value of $L$ to each of \texttt{PANOC}, \texttt{ZeroFPR}, and fast proximal gradient algorithms for computing their step lengths in our comparison.

The results are shown in \figurename~\ref{fig:alpha-plots}, \figurename~\ref{fig:splogl1-loss} and \figurename~\ref{fig:perf-prof}. In \figurename~\ref{fig:splogl1-loss}, we observe that \texttt{Prox-GGN-SCORE} reduces most of computational burden of the quasi-Newton method when $m+n_y<n$ and makes the method competitive with the first-order methods considered. However, as shown in both \figurename~\ref{fig:alpha-plots} and \figurename~\ref{fig:splogl1-loss}, \texttt{Prox-GGN-SCORE} is no longer preferred when $n<m+n_y$ and, by our experiments, the algorithm can run into computational issues when $n \ll m$. In this case (particularly for all of the real datasets that we use in this example), \texttt{Prox-N-SCORE} would be preferred and, as shown in the performance profile of \figurename~\ref{fig:perf-prof}, outperforms other tested algorithms in most cases, especially with $\alpha=1$.
\subsection{Sparse-group lasso}\label{ex:glasso}
\begin{sidewaystable}
	\centering
	\caption{Performance of \texttt{Prox-GGN-SCORE} (\texttt{alg.A}), \texttt{SSNAL} (\texttt{alg.B}), \texttt{Prox-Grad} (\texttt{alg.C}) and \texttt{BCD} (\texttt{alg.D}) on the sparse-group lasso problem \eqref{eq:sgl-example} for different values of $m$ and $n$. nnz stands for the number of nonzero entries of $x^\star $ and of the solutions found by the algorithms. MSE stands for the mean squared error between the true solution $x^\star $ and the estimated solutions.}
	\begin{adjustbox}{width=\textwidth}
		\setlength{\tabcolsep}{2.5pt} 
		\renewcommand{\arraystretch}{4} 
		\begin{tabular}{@{}|c|c|c|c|c|c|c|c|c|c|c|c|c|c|c|c|c|}
			\hline
			\multirow{2}{*}{$(m,n; nnz)$} &
			\multicolumn{4}{c|}{nnz} &
			\multicolumn{4}{c|}{Iteration\footnote{Number of ``outer'' iterations is displayed for \texttt{SSNAL} (\texttt{alg.B}). In an augmented Lagrangian method, most of the computational time is likely spent on the inner iterations.}} &
			\multicolumn{4}{c|}{Time [s]} &
			\multicolumn{4}{c|}{MSE} \\
			\cline{2-17}
			& \texttt{alg.A} & \texttt{alg.B} & \texttt{alg.C} & \texttt{alg.D} & \texttt{alg.A} & \texttt{alg.B} & \texttt{alg.C} & \texttt{alg.D} & \texttt{alg.A} & \texttt{alg.B} & \texttt{alg.C} & \texttt{alg.D} & \texttt{alg.A} & \texttt{alg.B} & \texttt{alg.C} & \texttt{alg.D} \\
			\hline
			$(500,2000; 19)$ & 19 & 198 & 19 & 19 & 161 & 62 & 5904 & 9690 & \textbf{2.81} & 4.65 & 13.39 & 3.55 & \text{2.9305E-09} & \text{5.3188E-08} & \text{2.5189E-07} & \text{8.0350E-06} \\
			$(500,4000; 36)$ & 36 & 39 & 36 & 36 & 253 & 140 & 10991 & 16790 & \textbf{8.44} & 39.51 & 51.91 & 11.60 & \text{1.4291E-08} & \text{4.3952E-08} & \text{1.1653E-06} & \text{1.7127E-05} \\
			$(500,5000; 45)$ & 45 & 45 & 45 & 45 & 530 & 111 & 13919 & 20830 & \textbf{16.60} & 35.57 & 90.05 & 18.53 & \text{2.6339E-07} & \text{6.0898E-08} & \text{2.0121E-06} & \text{2.1641E-05} \\
			$(1000,5000; 45)$ & 45 & 82 & 45 & 45 & 112 & 35 & 3051 & 9100 & \textbf{8.71} & 15.73 & 11.57 & 22.14 & \text{2.4667E-07} & \text{1.9757E-06} & \text{2.1747E-06} & \text{5.0779E-06} \\
			$(1000,7000; 65)$ & 65 & 65 & 65 & 65 & 185 & 82 & 7012 & 20870 & \textbf{30.26} & 148.07 & 42.45 & 70.08 & \text{4.5689E-07} & \text{2.2847E-08} & \text{4.0172E-06} & \text{1.8038E-05} \\
			$(1000,10000; 94)$ & 93 & 94 & 94 & 94 & 497 & 102 & 9879 & 29330 & \textbf{53.26} & 252.05 & 90.25 & 126.17 & \text{3.8421E-06} & \text{2.8441E-08} & \text{3.6320E-06} & \text{3.5855E-05} \\
			$(1000,12000; 112)$ & 112 & 113 & 113 & 164 & 663 & 68 & 21178 & 59360 & \textbf{166.15} & 194.40 & 221.26 & 373.50 & \text{1.5750E-05} & \text{4.6965E-08} & \text{7.3285E-06} & \text{5.9521E-05} \\
			\hline
	\end{tabular}\end{adjustbox}
	\label{tab:gl-results}
\end{sidewaystable}
\begin{figure}[t!]
	\centering
	\subfloat{%
		\resizebox*{7cm}{!}{\includegraphics{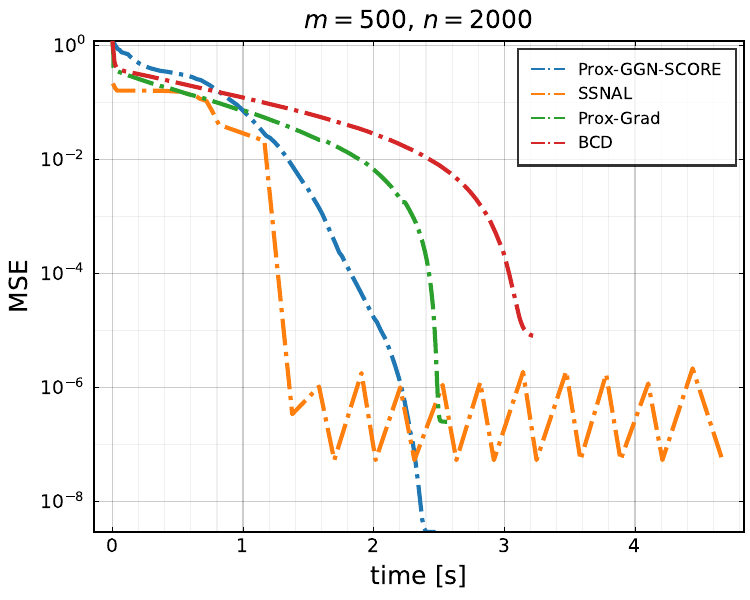}}}\hspace{5pt}
	\subfloat{%
		\resizebox*{7cm}{!}{\includegraphics{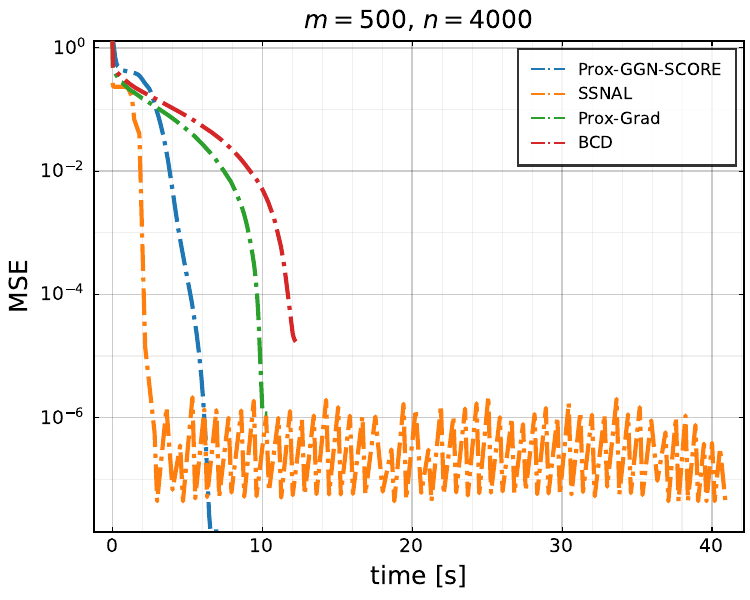}}}
	\vfil
	\subfloat{%
		\resizebox*{7cm}{!}{\includegraphics{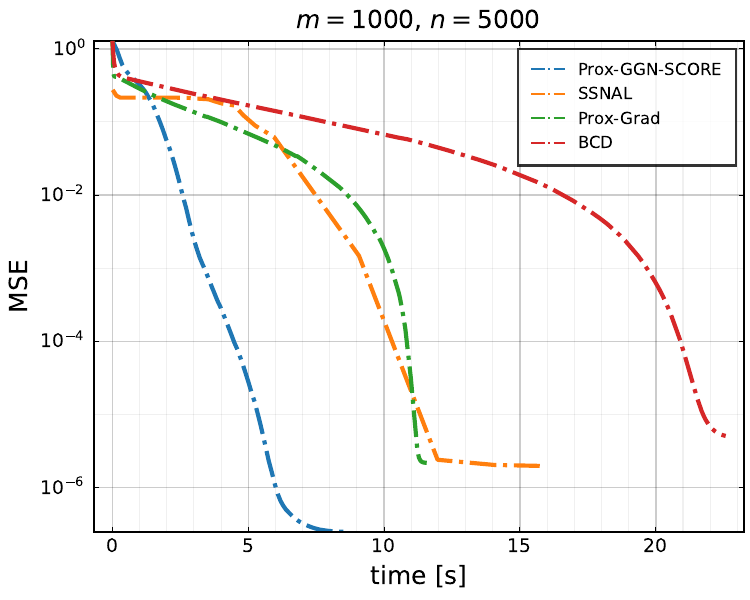}}}\hspace{5pt}
	\subfloat{%
		\resizebox*{7cm}{!}{\includegraphics{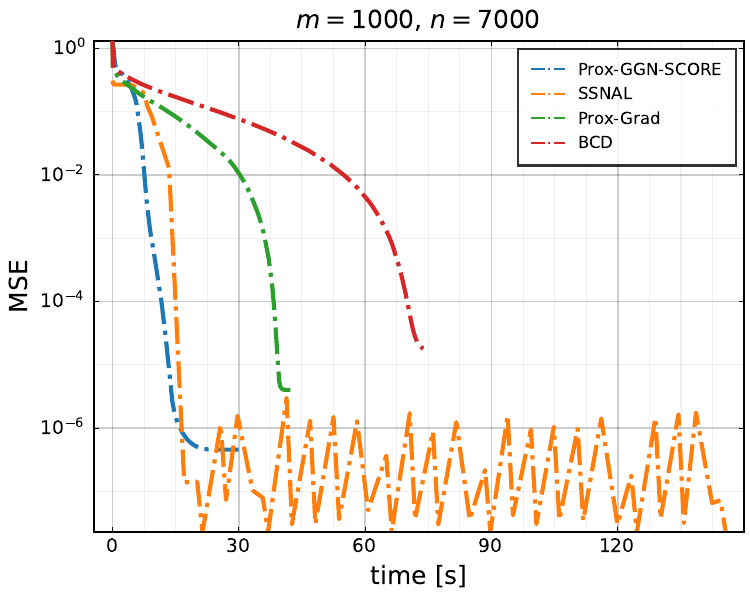}}}
	\caption{Mean squared error (MSE) between the estimates $x_k$ and the true coefficient $x^\star $ for \texttt{Prox-GGN-SCORE}, \texttt{SSNAL}, \texttt{Prox-Grad} and \texttt{BCD} on the sparse-group lasso problem \eqref{eq:sgl-example}.} \label{fig:sgl-plots}
\end{figure}
In this example, we consider the sparse-group lasso problem \eqref{eq:glasso}:
\begin{align}\label{eq:sgl-example}
	\min\limits_{x\in\rr^n} \calL(x) \coloneqq \underbrace{\frac{1}{2}\norm{Ax - y}^2}_{\eqqcolon f(x)} + \underbrace{\beta\norm{x}_1 + \beta_\calG \sum_{j\in \calG}\omega_j\|x^{(j)}\|}_{\eqqcolon g(x)}.
\end{align}
We use the common example used in the literature \cite{wang2014two,tibshirani2012strong}, which is based on the model $y = Ax^\star  + 0.01\epsilon\in \rr^{m\times 1}$, $\epsilon \sim \calN(0,1)$. The entries of the data matrix $A\in \rr^{m\times n}$ are drawn from the normal distribution with pairwise correlation $\mathrm{corr}(A^{(i)},A^{(j)}) = 0.5^{|i-j|}$, $\forall (i,j) \in \{1,\ldots,n\}^2$. We generate datasets for different values of $m$ and $n$ with $n$ satisfying $(n \mod n_g)=0$. In this problem, we want to further highlight the faster computational time achieved by the approximation in \texttt{Prox-GGN-SCORE}, so we consider only overparameterized models (\ie, with $m+n_y\le n$).

In this problem, the matrix $C$ in the reformulation \eqref{eq:structured-prob} is a diagonal matrix with row indices given by all pairs $(i,j) \in \{(i,j)|i\in j, i \in \{1,\ldots,n_g\}, j\in\calG\}$, and column indices given by $k\in\{1,\ldots,n_g\}$. That is,
\begin{align*}
	C^{((i,j),k)} = \begin{cases}
		\beta_\calG\omega_j &\quad\text{if }i=k,\\
		0 &\quad\text{otherwise}.
	\end{cases}
\end{align*}
We construct $x^\star $ in a similar way as \cite{ndiaye2016gap}: We fix $n_g=100$ and break $n$ randomly into groups of equal sizes with $0.1$ percent of the groups selected to be \emph{active}. The entries of the subvectors in the \emph{nonactive} groups are set to zero, while for the active groups, $\lceil\frac{n}{n_g}\rceil\times 0.1$ of the subvector entries are drawn randomly and set to $\sign(\xi) \times U$ where $\xi$ and $U$ are uniformly distributed in $[0.5,10]$ and $[-1,1]$, respectively; the remaining entries are set to zero. For the sake of fair comparison, each data and the associated initial vector $x_0$ are generated in Julia, and exported for the \texttt{BCD} implementation in Python and also for \texttt{SSNAL} in MATLAB.

For \texttt{Prox-GGN-SCORE}, \texttt{Prox-Grad} and \texttt{BCD}, we set $\beta = \tau_1\gamma\|A^\top y\|_\infty$, $\beta_\calG = (10-\tau_1)\gamma\|A^\top y\|_\infty$ with $\tau_1=0.9$ and $\gamma \in \{10^{-7},10^{-8}\}$. \texttt{SSNAL} can be made to return a solution estimate that has number of nonzero entries close to that of the true solution with a carefully tuned $\beta$ and simply setting $\beta_\calG = \beta$ (\cf, \cite[Table 1]{zhang2020efficient}). However, by our numerical experiments, \texttt{SSNAL} can be very sensitive to the choice of $\beta$ and $\beta_\calG$ if the goal is to have a reasonable convergence to the true solution with the correct within-group sparsity in the solution estimate. After a careful tuning, and for the sake of fair comparison, we set $\beta = \tau_1\gamma\|A^\top y\|_\infty$ and $\beta_\calG = \|A^\top y\|_\infty$ with $\gamma = 10^{-5}$ and $\tau_1 \in \{4,5,10,12\}$ (depending on the problem size) for \texttt{SSNAL}. For each group $j$, the parameter $\omega_j$ is set to the standard value $\sqrt{n_j}$ \cite{friedman2010note,simon2013sparse}, where $n_j=\card(j)$. For fairness, the estimate $\alpha_k = 1/L$ with $L = \lambda_{max} (A^\top A)$ is used in the proximal gradient and \texttt{SSNAL} algorithms.

We set $\mu$ to $1.2$ for $m=500$, $n=2000$, $2.0$ for $m=1000$, $n=12000$, and to $1.6$ in the remaining setups. The simulation results are shown in \tablename~\ref{tab:gl-results} and \figref{fig:sgl-plots}. As shown, \texttt{Prox-GGN-SCORE} terminates faster than \texttt{SSNAL}, \texttt{Prox-Grad} and \texttt{BCD} algorithms in most cases with the correct number of nonzero entries in its solution estimates. Additionally, the results further highlight the computational benefits of \texttt{Prox-GGN-SCORE} for overparameterized problems.
\begin{figure}[t!]
	\centering
	\includegraphics[width=\linewidth]{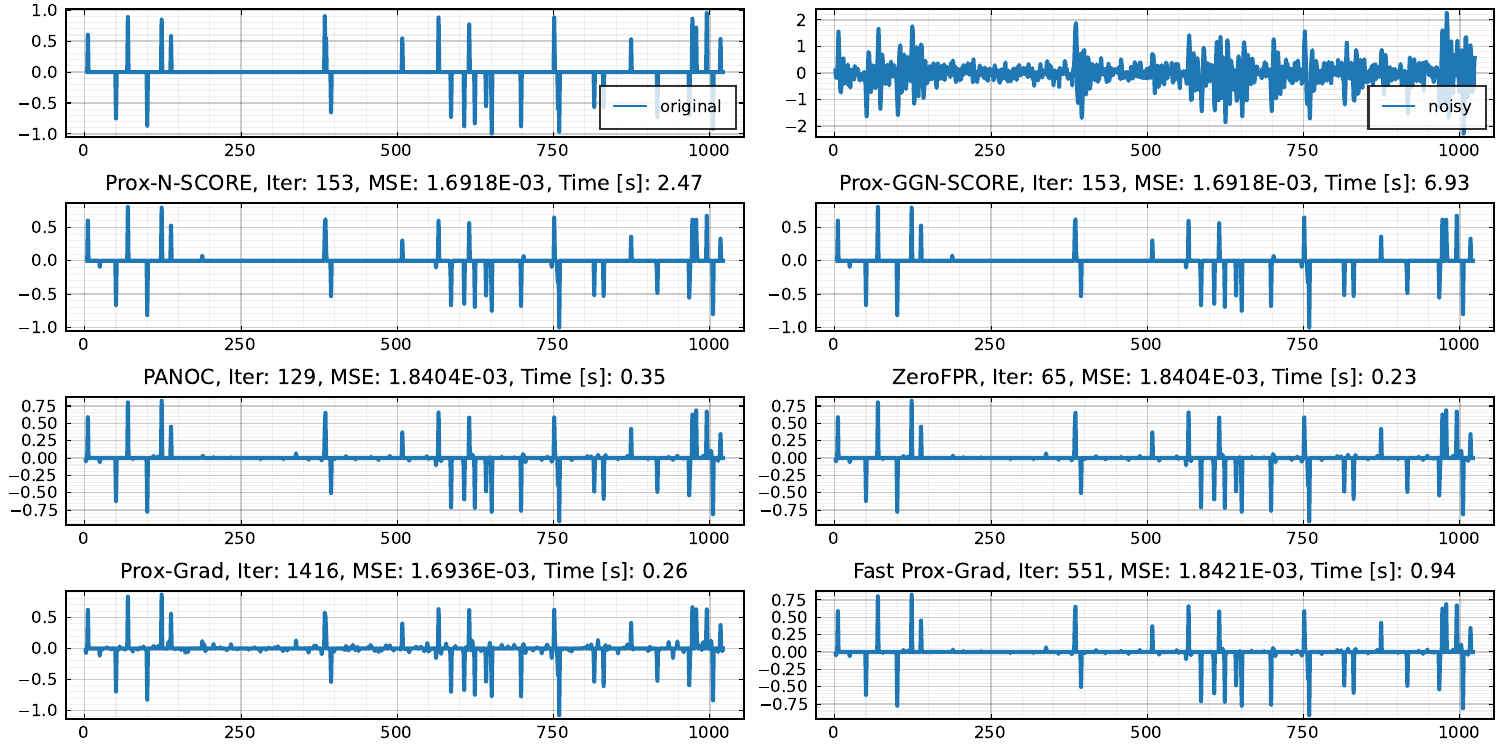}
	\caption{Sparse deconvolution via $\ell_1$-regularized least squares \eqref{eq:lsqexample} using \texttt{Prox-N-SCORE}, \texttt{Prox-GGN-SCORE}, \texttt{PANOC}, \texttt{ZeroFPR}, proximal gradient, and fast proximal gradient algorithms with $n=1024$.}
	\label{fig:deconv-l1}
\end{figure}
\begin{figure}[t!]
	\centering
	\includegraphics[width=\linewidth]{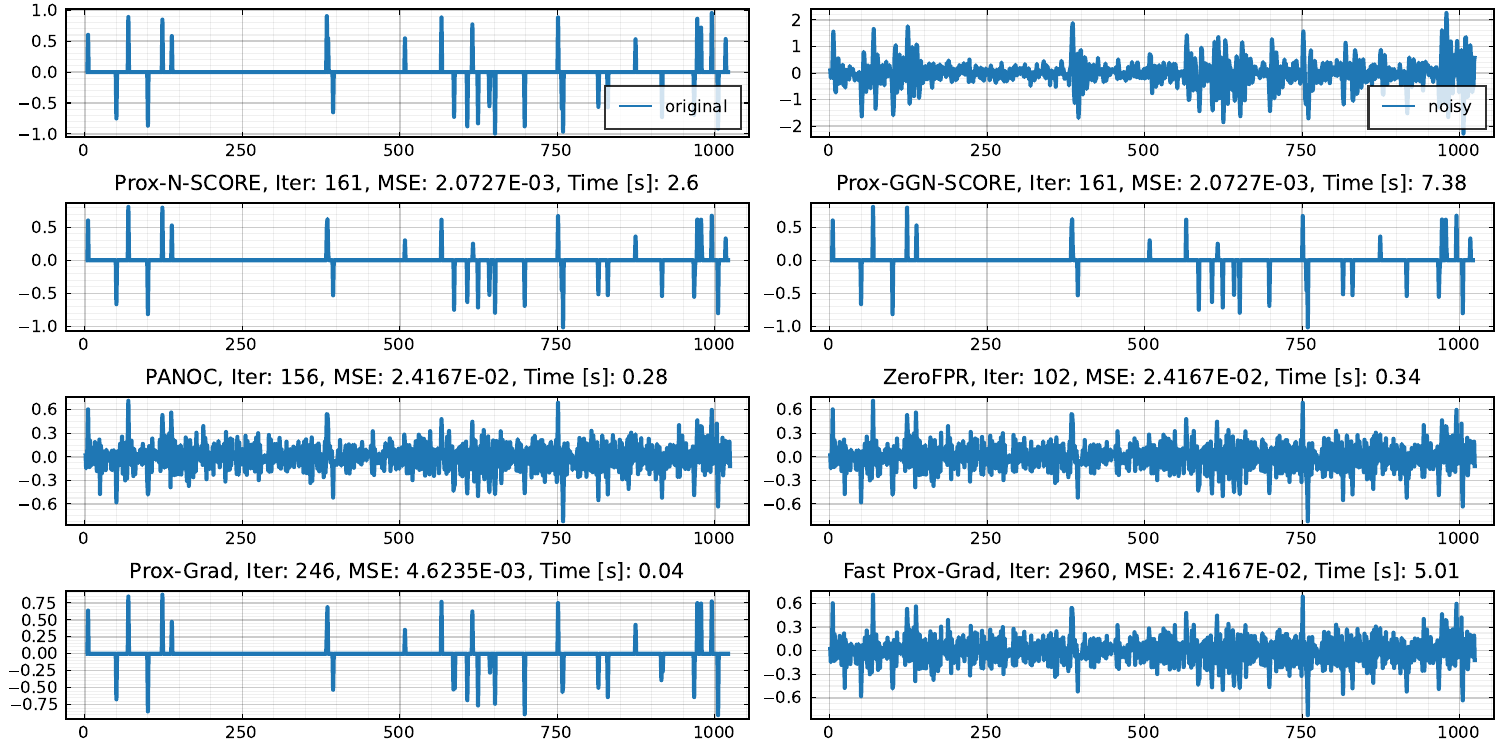}
	\caption{Sparse deconvolution via $\ell_2$-regularized least squares \eqref{eq:lsqexample} using \texttt{Prox-N-SCORE}, \texttt{Prox-GGN-SCORE}, \texttt{PANOC}, \texttt{ZeroFPR}, proximal gradient, and fast proximal gradient algorithms with $n=1024$.}
	\label{fig:deconv-l2}
\end{figure}
\subsection{Sparse deconvolution}\label{ss:lsqexample}
In this example, we consider the problem of estimating the unknown sparse input $x$ to a linear system, given a noisy output signal and the system response. That is,
\begin{align}
	\min\limits_{x\in\rr^n} \calL(x) \coloneqq \underbrace{\frac{1}{2}\norm{Ax - y}^2}_{\eqqcolon f(x)} + \beta \|x\|_p, \label{eq:lsqexample}
\end{align}
where $A\in \rr^{n\times n}$ and $y\in\rr^{n\times 1}$ are given data about the system which we randomly generate according to \cite[Example F]{selesnick2014sparse}.

We solve with both $\ell_1$ ($p=1$) and $\ell_2$ ($p=2$) regularizers, and set $\beta = 10^{-3}$. We set $\mu = 5\times 10^{-2}$ in the smooth approximation $g_s$ of $g$. We estimate $L = \lambda_{max} (A^\top A)$ and set $\alpha_k = 1/L$ in the proximal gradient algorithm. Again, for fairness, we provide this value of $L$ to each of \texttt{PANOC}, \texttt{ZeroFPR}, and fast proximal gradient procedures in our comparison. The simulation results are displayed in \figurename~\ref{fig:deconv-l1} and \figurename~\ref{fig:deconv-l2}. While \texttt{Prox-GGN-SCORE} and \texttt{Prox-N-SCORE} sometimes use more computational time in this problem, they provide better solution quality with smaller reconstruction error than the other tested algorithms, which is more desirable for signal reconstruction problems.
\section{Conclusion}\label{sec:conclusions}
In this paper, we introduced a self-concordant regularization framework for proximal quasi-Newton methods that solves large-scale convex composite optimization problems while preserving the structure induced by nonsmooth regularizers. Two algorithms are studied: a proximal Newton algorithm (\texttt{Prox-N-SCORE}) and a proximal generalized Gauss-Newton algorithm (\texttt{Prox-GGN-SCORE}). Both algorithms share an adaptive step length rule that eliminates the need for line search or trust-region subroutines, and they employ a diagonal variable metric derived from the smooth regularization. These design choices guarantee global convergence and yield favorable local behaviour under standard regularity assumptions. The \texttt{Prox-GGN-SCORE} variant relies on a low-rank approximation of the Hessian inverse that exploits the structure of prediction models (\eg, in machine learning). This makes it especially effective for overparameterized regimes where the number of decision variables exceeds the number of observations, allowing the method to scale to high-dimensional problems without forming full matrix inverses. Future work will focus on adaptive selection of the smoothing parameter, a theoretical analysis of how self-concordant smoothing influences optimization dynamics and generalization in scientific machine learning settings, and the derivation of explicit complexity estimates for both algorithms.

\bibliographystyle{tfs}
\bibliography{references.bib}

\end{document}